\documentclass[preprint,12pt]{elsarticle}

\usepackage{amssymb}
\usepackage{amsmath}
\usepackage{amsfonts}
\usepackage{graphicx,amsmath,amsfonts,amssymb,amscd,amsthm}
\usepackage{epsfig}
\usepackage{epsf}
\biboptions{sort&compress}
\newtheorem{theorem}{Theorem}

\newtheorem{definition}{Definition}
\newtheorem{lemma}{Lemma}

\newtheorem{remark}{Remark}
\newtheorem{problem}{Problem}

\newcommand{\x}{\mbox{\boldmath $x$}}

\newcommand{\be}{\begin{equation}}
\newcommand{\ee}{\end{equation}}
\newcommand{\ba}{\begin{align}}
\newcommand{\ea}{\end{align}}

\newcommand{\bea}{\begin{array}}
\newcommand{\eea}{\end{array}}

\def\EquationsBySection{\def\theequation
{\thesection.\arabic{equation}}%
\@addtoreset{equation}{section}}

\newcommand\old[1]{}
\newcommand{\pend}{\hfill \thicklines \framebox(6.6,6.6)[l]{}}
\renewenvironment{proof}{\noindent {\it  Proof.} \rm}{\pend}
\newtheorem{proposition}{Proposition}
\newtheorem{example}{Example}

\usepackage{indentfirst}
\EquationsBySection \makeatother
\usepackage{indentfirst}
\usepackage{floatrow}
\newfloatcommand{capbtabbox}{table}[][\FBwidth]
\journal{}
\begin{document}

\begin{frontmatter}



\title{Anticipated backward stochastic Volterra integral equations
and their applications to nonzero-sum stochastic differential games}


\author[label1]{Bixuan Yang}
\ead{bixuanyang@126.com}
\author[label2]{Tiexin Guo \corref{cor1}}
\ead{tiexinguo@csu.edu.cn}
\cortext[cor1]{Corresponding author.}

\address[label1]{\small\it School of Mathematics and Statistics, Hunan First Normal University, Changsha 410205, Hunan, PR China}
\address[label2]{\small\it School of Mathematics and Statistics, Central South University, Changsha 410083, Hunan, PR China}

%

\begin{abstract}
In [J. Wen, Y. Shi, Stat. Probab. Lett. 156 (2020) 108599]
the authors first introduced a kind of anticipated backward stochastic
Volterra integral equations (anticipated BSVIEs, for short).
By virtue of the duality principle, it is found in this paper that the
anticipated BSVIEs can be applied to the study of stochastic differential games.
Naturally, in order to develop
the related theories and applications of BSVIEs, in this paper
we deeply investigate a more general class of anticipated BSVIEs whose generator includes both
pointwise and average time-advanced functions.
In theory, the well-posedness and the comparison theorem
of anticipated BSVIEs are established,
and some regularity results of
adapted M-solutions are proved by applying Malliavin calculus,
which cover the previous results for BSVIEs.
Further, using linear anticipated BSVIEs as the adjoint equation,
we present the maximum principle for the nonzero-sum differential
game system of stochastic delay
Volterra integral equations (SDVIEs, for short) for the first time.
As one of the applications of the principle,
a Nash equilibrium point of the linear-quadratic differential game problem of SDVIEs is obtained.
\end{abstract}


\end{frontmatter}

\section{Introduction and motivation}\label{sec1}
It is well-known that the classical nonlinear backward stochastic differential
equation (BSDE, for short)
was initiated by Pardoux and Peng \cite{Pardoux1990} in $1990$, and
they completely solved the existence and
uniqueness of adapted solutions when the generator $g(\cdot)$ satisfies the Lipschitz condition.
Since then, the theory of BSDEs
has been widely applied in various fields such as
stochastic optimal control,
stochastic differential games, mathematical finance, engineering,
insurance and partial differential equations (PDEs, for short),
see e.g. \cite{Jaber2021,Buckdahn2017,Karoui1997,Liu2021,Moon2020,Tian2021,Wang2024a}.
Specifically, based on the regularity of adapted solutions
to BSDEs, the famous Feynman-Kac formula and
the related PDEs can be used for option pricing
in the financial markets (see \cite{Nualart2001}).
More details about the theory of BSDEs
can be referred to in \cite{Agram2022,Hu2024,Li2026,Yu2022}.

By virtue of the duality with the stochastic differential delay equations
(SDDEs, for short),
Peng and Yang \cite{Peng2009} further presented anticipated backward stochastic differential
equations (ABSDEs, for short) as follows:
\begin{align}\label{49}
\left\{\begin{array}{lcl} d Y(t) = - g\left( t, Y(t), Z(t), Y(t+\delta (t)),
Z(t+\zeta(t))  \right)dt + Z(t) dW(t),\ \ \
  t\in [0, T],\\
Y(t) = \varphi (t), \ \ t\in [T, T+K],\\
Z(t) = \eta (t), \ \ t\in [T, T+K].
\end{array}
\right.
\end{align}
The main feature of this type of ABSDE (\ref{49}) is that
the generator $g(\cdot)$ depends on both the present and the future values of the adapted solution $(Y(\cdot),Z(\cdot))\in L_{\mathbb{F}}^{2} (0,T+K; \mathbb{R}^n)\times L_{\mathbb{F}}^{2} (0,T+K; \mathbb{R}^{n\times m})$ with $\varphi (\cdot)\in L_{\mathbb{F}}^{2} (T,T+K; \mathbb{R}^n)$ and $\eta(\cdot)\in L_{\mathbb{F}}^{2} (T,T+K; \mathbb{R}^{n\times m})$.
Thanks to its profound theoretical significance and extensive
applications, the ABSDEs have attracted many scholars' interest.
Especially, as the adjoint equation of stochastic control systems
with both pointwise and average delays,
the ABSDEs with both pointwise and average time-advanced functions
are one of the most important types.
They have been successfully applied to mathematical finance,
such as the problem of optimal
consumption rates in a market with delay
(see \cite{Agram2013}).
Up to now, the works following this line can be found in
\cite{Chen2026,Meng2015,Meng2025,Wang2025,Wu2025}.
For this, this paper will introduce such typical ABSDEs with both pointwise and average
time-advanced functions
into the backward stochastic Volterra integral equations (BSVIEs, for short)
to enrich and develop the related theories and applications of BSVIEs.

The remarkable research on
BSVIEs began with the work of
Yong \cite{Yong2008}, and their forms are as follows:
\begin{equation*}
\begin{aligned}
Y(t)=\varphi(t)+\int_{t}^{T} g(t, s, Y(s), Z(t, s), Z(s, t))ds
-\int_{t}^{T} Z(t, s) dW(s), ~~t\in[0, T].
\end{aligned}
\end{equation*}
It is found that $\{\int_{t}^{T} Z(t, s) dW(s)$, $0\leq t\leq T\}$
is not a martingale so that It\^{o}'s formula does not hold for BSVIEs,
so the study of the theory of BSVIEs requires new tools.
Yong \cite{Yong2008} and Wang and Yong \cite{Wang2015} systematically
studied the properties of BSVIEs by introducing an adapted M-solution, and
successfully generalized the regularity of BSDEs to the BSVIE case via Malliavin calculus.
The theory of BSVIEs
has been applied in stochastic control theory \cite{Shi2013,Wang2022,Wang2023},
stochastic PDEs \cite{Wang2019} and risk measurement \cite{Miao2021,Wang2021b,Yang2024},
and thus also have attracted increasing attention from scholars.
Using Malliavin calculus, Wang \cite{Wang2021a} further discussed the Feynman-Kac formula for extended BSVIEs
through some regularity results.
Wen and Shi \cite{Wen2020} were devoted to a kind of anticipated BSVIEs (ABSVIEs, for short)
with form
\begin{align}\label{56}
\left\{\begin{array}{lcl} Y(t) = \varphi (t)+\int_{t}^{T} g\left(
t, s, Y(s), Z(t, s), Z(s, t), Y(s+\delta (s)),
Z(t, s+\zeta(s)), \right.\\
~~~~~~~~~\left. Z(s+\zeta(s), t) \right)ds-\int_{t}^{T} Z(t, s) dW(s),\ \ \
  t\in [0, T],\\
Y(t) = \varphi (t), \ \ t\in [T, T+K],\\
Z(t, s) = \eta(t, s), \ \ (t, s)\in [0, T+K]^2 \backslash [0, T]^2,
\end{array}
\right.
\end{align}
and presented the well-posedness
and the comparison theorem of them. Recently, Hamaguchi \cite{Hamaguchi2023}
mainly dealt with the maximum principle for infinite horizon optimal
control problem of stochastic delay Volterra integral equations
(SDVIEs, for short). Wang and Zheng \cite{Wang2024b}
investigated the singular BSVIE in infinite dimensional spaces.
One of the purposes of this paper is to extend the results of \cite{Wen2020} to
a more general ABSVIE.

As one of the most vigorous branches of optimal controls,
stochastic differential games have been a powerful tool for
modeling stochastic dynamic systems involving multiple decision
makers, and played a central role in biology, economics and finance.
For example, Savku and Weber \cite{Savku2022} studied the
optimal investment problems by
using stochastic game approaches
within the framework of behavioral finance.
The two major ways to solve stochastic differential game problems
are Bellman's dynamic programming method and Pontryagin's
maximum principle. Relevant theory of stochastic differential games
can be found in
\cite{Hamadene2020,Moon2025,Nie2025,Wang2018,
Wu2024, Yang2019}.
Since a BSVIE is not Markovian, it is necessary to apply Pontryagin's
maximum principle to deal with the corresponding nonzero-sum differential games.

As shown in Example \ref{ex1} of our paper, it is found that the equations
as in (\ref{56}) can be applied to the study of stochastic differential games. Whereas, in order to solve the nonzero-sum differential game problem of SDVIE (\ref{29}), it is necessary to
establish a more general duality principle.
For this, due to the wide applicability of our method of proofs,
and also to provide a rich theoretical basis for further research on the Feynman-Kac formula,
this paper introduces and
studies the properties of adapted M-solutions to general ABSVIEs as in (\ref{41}) below,
and extends the applications of BSVIEs
from optimal controls to nonzero-sum stochastic differential games.
More precisely, besides the application mentioned above, another topic of this paper is to consider the following class of ABSVIEs:
\begin{align}\label{41}
\left\{\begin{array}{lcl} Y(t) = \varphi (t)+\int_{t}^{T} g\left( \Lambda (t,s) \right)ds
 -\int_{t}^{T} Z(t, s) dW(s),\ \ \
  t\in [0, T],\\
Y(t) = \varphi (t), \ \ t\in [T, T+K],\\
Z(t, s) = \eta(t, s), \ \ (t, s)\in [0, T+K]^2 \backslash [0, T]^2,
\end{array}
\right.
\end{align}
where
\begin{align*}
\Lambda (t,s)&:= \left( t, s, Y(s), Z(t, s), Z(s, t), Y(s+\delta (s)),
Z(t, s+\zeta(s)), Z(s+\zeta(s), t),\right.\\
&~~~ \left. \int_{s}^{s+\delta (s)} e^{\lambda (s-\theta)} Y(\theta)d\theta,
\int_{s}^{s+\zeta (s)} e^{\lambda (s-\theta)} Z(t, \theta)d\theta,  \int_{s}^{s+\zeta (s)} e^{\lambda (s-\theta)} Z(\theta, t)d\theta  \right).
\end{align*}
Here, $Y(s+\delta (s))$, $Z(t, s+\zeta(s))$ and $Z(s+\zeta(s), t)$ are said to be
pointwise time-advanced.
$\int_{s}^{s+\delta (s)} e^{\lambda (s-\theta)} Y(\theta)d\theta$,
$\int_{s}^{s+\zeta (s)} e^{\lambda (s-\theta)} Z(t, \theta)d\theta$ and
$\int_{s}^{s+\zeta (s)} e^{\lambda (s-\theta)} Z(\theta, t)d\theta$
are said to be average time-advanced.
$\varphi (\cdot)$ and $\eta(\cdot, \cdot)$ are the terminal conditions of (\ref{41}),
and $g(\cdot)$ is the generator of (\ref{41}).

Although the work of this paper considerably benefits from the recent advance
in \cite{Wang2015,Wen2020,Yong2008},
the generator $g(\cdot)$ of ABSVIE (\ref{41}) contains both pointwise
time-advanced functions and average time-advanced functions, which makes the analysis more challenging.
Just motivated by the idea of \cite{Wang2015,Yong2008}, in this paper
we first establish the well-posedness and the comparison theorem
of ABSVIEs,
which cover the previous results for BSVIEs.
There is no need to assume $\beta>0$ as in \cite{Wen2020} in proving the well-posedness to ABSVIE (\ref{41})
so that the conditions of Theorem 3.2 in \cite{Wen2020} can be weakened.
Furthermore, based on the contractive map defined in Theorem \ref{th1}, we apply Malliavin calculus to the ABSVIEs and
directly prove their
differentiability that seems to be new to
our best knowledge. This means that
our result would prompt us to further consider the corresponding Feynman-Kac formula
in the near future.
Finally,
by virtue of the
duality principle, Pontryagin's type maximum principle for
nonzero-sum differential games of SDVIEs
is presented for the first time.
As an example, we get a Nash equilibrium point
for the linear-quadratic differential game problem of SDVIEs,
also see \cite{Guo2025} for the related study.
This result can make up for the lack of the applications of BSVIEs in
stochastic differential games.

The rest of this paper is organized as follows. In Section \ref{sec2}, we
introduce some preliminaries.
In Sections \ref{sec3}-\ref{sec4}, we are concerned with the well-posedness
and the comparison theorem of ABSVIEs, respectively.
Section \ref{sec5} is devoted to the study of regularity results for the adapted M-solution to ABSVIEs.
Finally, an application to the nonzero-sum differential game
driven by SDVIEs is presented in Section \ref{sec6}.

\section{Preliminaries}\label{sec2}
Throughout this paper,
let $(\Omega, \mathcal{F}, \mathbb{F},
\mathbb{P})$ be a complete filtered probability space
on which an $m$-dimensional standard Brownian motion $\{W(t); 0\leq t <\infty\}$ is defined
with $\mathbb{F}=\{\mathcal{F}_t\}_{t\geq 0}$ being its natural filtration
augmented by all the $\mathbb{P}$-null sets in $\mathcal{F}$.
Let $T>0$ be a finite time horizon
and $K\geq 0$ be a constant.
Denote by
\begin{equation*}
\begin{aligned}
\Delta=\{(t,s)\in[0, T]^2~|~t\leq s\}~~and~~\Delta^c=\{(t,s)\in[0, T]^2~|~t> s\}.
\end{aligned}
\end{equation*}
In addition, $\mathcal{B}(\mathbb{H})$ denotes the Borel $\sigma$-field of metric space $\mathbb{H}$,
$\mathbb{K}$ and $\widetilde{\mathbb{K}}$ denote the Euclidean spaces, and
for any $t\in[0,T+K]$, we employ the following notations:
\begin{equation*}
\begin{aligned}
&C^k (\mathbb{K};\widetilde{\mathbb{K}})=\left\{\chi: \mathbb{K}\mapsto\widetilde{\mathbb{K}}~|~
\chi (\cdot)~\mbox{is}~k\mbox{-th continuously differentable}\right\},\\
&C_b^k (\mathbb{K};\widetilde{\mathbb{K}})=\left\{\chi\in C^k (\mathbb{K};\widetilde{\mathbb{K}})~|~
\mbox{for any}~0\leq i \leq k,~\mbox{the}~i\mbox{-th order partial
derivative of}~\chi (\cdot)\right.\\
&~~~~~~~~~~~~~~~~\left.\mbox{is bounded}\right\},\\
&L^{2} (0,T+K; \mathbb{K})=\left\{\mu: [0,T+K]\mapsto\mathbb{K}~|~\mu(\cdot) ~\mbox{is}~\mathcal{B}([0,T+K])\mbox{-measurable},\right.\\
&\left.~~~~~~~~~~~~~~~~~~~~~~~~\int_{0}^{T+K}|\mu(t)|^2 dt<\infty\right\},
\end{aligned}
\end{equation*}
\begin{equation*}
\begin{aligned}
&L_{\mathcal{F}_t}^{2} (\Omega; \mathbb{K})=\left\{\xi: \Omega\mapsto\mathbb{K}~|~\xi~\mbox{is} ~\mathcal{F}_t\mbox{-measurable},
~\mathbb{E}\left[|\xi|^2\right]<\infty\right\},\\
&L_{\mathcal{F}_T}^{2} (0,T; \mathbb{K})=\left\{\mu: \Omega\times[0,T]\mapsto\mathbb{K}~|~\mu(\cdot) ~\mbox{is}~\mathcal{F}_{T}\otimes\mathcal{B}([0,T])\mbox{-measurable},\right.\\
&\left.~~~~~~~~~~~~~~~~~~~~\mathbb{E}\left[\int_{0}^{T}|\mu(t)|^2 dt\right]<\infty\right\},\\
&\hat{L}_{\mathcal{F}_{T+K}}^{2} (0,T+K; \mathbb{K})=\left\{\mu: \Omega\times[0,T+K]\mapsto\mathbb{K}~|~\mu(\cdot) ~\mbox{is}~\mathcal{F}_{T+K}\otimes\mathcal{B}([0,T+K])\mbox{-measurable},  \right.\\
&~~~~~~~~~~~~~~~~~~~~~~~~~~~~\mbox{and}~\mu(t)~\mbox{is}~\mathcal{F}_{T\vee t}\mbox{-measurable}~\mbox{for all}~t\in [0,T+K],\\
&\left.~~~~~~~~~~~~~~~~~~~~~~~~~~~~\mathbb{E}\left[\int_{0}^{T+K}|\mu(t)|^2 dt\right]<\infty\right\},\\
&\hat{L}_{\mathcal{F}_{T+K}}^{\infty} (0,T+K; \mathbb{K})=\left\{\mu: \Omega\times[0, T+K]\mapsto\mathbb{K}~|~\mu(\cdot) ~\mbox{is}~\mathcal{F}_{T+K}\otimes\mathcal{B}([0,T+K])\mbox{-measurable},
\right.\\
&~~~~~~~~~~~~~~~~~~~~~~~~~~~~
\mbox{and}~\mu(t)~\mbox{is}~\mathcal{F}_{T\vee t}\mbox{-measurable}~\mbox{for all}~t\in [0,T+K],\\
&\left.~~~~~~~~~~~~~~~~~~~~~~~~~~~\sup\limits_{t \in [0, T+K]} \mathbb{E}\left[|\mu(t)|^2\right] <\infty\right\},\\
&L_{\mathbb{F}}^{2} (0,T+K; \mathbb{K})=\left\{\mu (\cdot) \in \hat{L}_{\mathcal{F}_{T+K}}^{2} (0,T+K; \mathbb{K})~|~\mu(\cdot)~ \mbox{is}~\mathbb{F}\mbox{-adapted}\right\},\\
&L_{\mathbb{F}}^{2} (\Delta; \mathbb{K})=\left\{\gamma: \Omega\times\Delta\mapsto\mathbb{K}~|
~\gamma (\cdot,\cdot)~\mbox{is}~\mathcal{F}_{T}\otimes\mathcal{B}(\Delta)\mbox{-measurable, and for any}~t\in [0, T],  \right.\\
&~~~~~~~~~~~~~~~~ \left. \gamma(t,s)~ \mbox{is}~\mathcal{F}_{s}\mbox{-measurable for all}~s\in [t,T],~ \mathbb{E}\left[\int_{0}^{T}\int_{t}^{T}|\gamma(t,s)|^2 dsdt\right]<\infty\right\},\\
&L_{\mathbb{F}}^{2} (\Delta^c; \mathbb{K})=\left\{\gamma: \Omega\times\Delta^c \mapsto\mathbb{K}~|
~\gamma (\cdot,\cdot)~\mbox{is}~\mathcal{F}_{T}\otimes\mathcal{B}(\Delta^c)\mbox{-measurable, and for any}~t\in [0, T], \right.\\
&~~~~~~~~~~~~~~~~~\left.  \gamma(t,s)~ \mbox{is}~\mathcal{F}_s\mbox{-measurable for all}~s\in [0,t], ~\mathbb{E}\left[\int_{0}^{T}\int_{0}^{t}|\gamma(t,s)|^2 dsdt\right]<\infty\right\},\\
&L^2 (0,T+K; L_{\mathbb{F}}^2 (0,T+K; \mathbb{K}))=\left\{\gamma: \Omega\times[0,T+K]^2  \mapsto\mathbb{K}~|
~\gamma (\cdot,\cdot)~\mbox{is}~\mathcal{F}_{T+K}\otimes\mathcal{B}([0,T+K]^2)\mbox{-}\right.\\
&~~~~~~~~~~~~~~~~~~~~~~~~~~~~~~~~~~~~~~~~~~~
\left. \mbox{measurable, and for any}~t\in [0, T+K],~\gamma(t,\cdot)\in \right.\\
&~~~~~~~~~~~~~~~~~~~~~~~~~~~~~~~~~~~~~~~~~~~
\left. L_{\mathbb{F}}^2 (0,T+K; \mathbb{K}), ~\mathbb{E}\left[\int_{0}^{T+K}\int_{0}^{T+K}|\gamma(t,s)|^2 dsdt\right]<\infty\right\}.
\end{aligned}
\end{equation*}

Obviously, for any $0\leq G\leq H \leq T+K$ and $0\leq\delta,\delta_i\leq K~(i=1,2)$, $L_{\mathcal{F}_H}^{2} (G,H; \mathbb{K})$, $L_{\mathbb{F}}^{2} (G,H; \mathbb{K})$, $L_{\mathbb{F}}^{2} (-\delta,T; \mathbb{K})$, $L_{\mathbb{F}}^{2} (-\delta_i,0; \mathbb{K})$,
$L^2 (G,H; L_{\mathbb{F}}^2 (G,H; \mathbb{K}))$, $L^\infty(0,T; L_{\mathbb{F}}^{\infty} (0,T; \mathbb{K}))$
have similar definitions.

\section{Well-posedness of ABSVIEs}\label{sec3}
In this section, in order to prove the well-posedness of adapted M-solutions to ABSVIE (\ref{41}),
we first consider a special kind of BSVIEs as follows:
\begin{equation}\label{58}
\begin{aligned}
Y(t) = \varphi (t)+\int_{t}^{T} g( t, s)ds
-\int_{t}^{T} Z(t, s) dW(s),\ \
   t\in [0, T],
\end{aligned}
\end{equation}
where $\varphi (\cdot)\in L_{\mathcal{F}_T}^{2} (0,T; \mathbb{R}^n)$, and $g(\cdot,\cdot)\in L_{\mathbb{F}}^{2} (\Delta; \mathbb{R}^n)$.
Naturally, let us introduce the definition of adapted M-solutions to (\ref{58}).
\begin{definition}\label{def2}
A pair of processes $(Y(\cdot),Z(\cdot,\cdot))\in L_{\mathbb{F}}^{2} (0,T; \mathbb{R}^n)\times L^2 (0,T; L_{\mathbb{F}}^2 (0,T; \mathbb{R}^{n\times m}))$ is called an adapted M-solution of BSVIE (\ref{58}) if (\ref{58}) holds in the usual It\^{o}'s sense for almost everywhere all $t\in [0, T]$, and for any $A\in [0,T]$, the following relation holds:
\begin{equation*}
\begin{aligned}
Y(t)=\mathbb{E}[Y(t)|\mathcal{F}_A]+\int_A^t Z(t,s) dW(s),~~~a.e.~t\in [A,T].
\end{aligned}
\end{equation*}
\end{definition}
By virtue of Definition \ref{def2}, we prove the existence and uniqueness of the adapted M-solution to (\ref{58}).
\begin{lemma}\label{lem1}
For any given $\varphi (\cdot)\in L_{\mathcal{F}_T}^{2} (0,T; \mathbb{R}^n)$ and $g(\cdot,\cdot)\in L_{\mathbb{F}}^{2} (\Delta; \mathbb{R}^n)$,
there exists a unique adapted M-solution $(Y(\cdot),Z(\cdot,\cdot))\in L_{\mathbb{F}}^{2} (0,T; \mathbb{R}^n)\times L^2 (0,T; L_{\mathbb{F}}^2 (0,T; \mathbb{R}^{n\times m}))$ of (\ref{58}).
Moreover, there exists a constant $C>0$ such that
\begin{equation}\label{4}
\begin{aligned}
\mathbb{E}\left[\int_0^T  |Y (t)|^2 dt+\int_0^T \int_t^T |Z (t,s)|^2 dsdt\right]
\leq C \mathbb{E}\left[\int_0^{T} |\varphi (t)|^2 dt
+\int_0^T \left(\int_t^T  |g (t,s)| ds\right)^2 dt\right].
\end{aligned}
\end{equation}
For $i=1,2$, if $\varphi_i (\cdot)\in L_{\mathcal{F}_T}^{2} (0,T; \mathbb{R}^n)$ and
$g_i(\cdot,\cdot)\in L_{\mathbb{F}}^{2} (\Delta; \mathbb{R}^n)$, then the corresponding adapted M-solution $(Y_i (\cdot),Z_i (\cdot,\cdot))\in L_{\mathbb{F}}^{2} (0,T; \mathbb{R}^n)\times L^2 (0,T; L_{\mathbb{F}}^2 (0,T; \mathbb{R}^{n\times m}))$ of (\ref{58}) satisfies the following stability
estimate:
\begin{equation}\label{5}
\begin{aligned}
&~~~\mathbb{E}\left[\int_0^T  |Y_1 (t)-Y_2 (t)|^2 dt+\int_0^T \int_t^T |Z_1 (t,s)-Z_2 (t,s)|^2 dsdt\right] \\
&\leq C \mathbb{E}\left[\int_0^{T} |\varphi_1 (t)-\varphi_2 (t)|^2 dt
+\int_0^T \left(\int_t^T  \left| g_1 (t,s)-g_2 (t,s)\right| ds\right)^2 dt
 \right].
\end{aligned}
\end{equation}
\end{lemma}
\begin{proof}
For any fixed $t\in [0, T]$, we investigate the following BSDE:
\begin{equation}\label{50}
\begin{aligned}
\mathcal{Y}_t(r) = \varphi (t)+\int_{r}^{T} g(t, s)ds
-\int_{r}^{T} \mathcal{Z}_t (s) dW(s),\ \
   r\in [t, T].
\end{aligned}
\end{equation}
Applying Theorem 4.1 in \cite{Pardoux1990}, we can obtain that BSDE (\ref{50}) admits a unique adapted solution
$(\mathcal{Y}_t (\cdot), \mathcal{Z}_t (\cdot))\in L_{\mathbb{F}}^{2} (t,T; \mathbb{R}^n)\times L_{\mathbb{F}}^2 (t,T; \mathbb{R}^{n\times m})$, and the following estimate holds:
\begin{equation}\label{6}
\begin{aligned}
\mathbb{E}\left[ \int_t^T |\mathcal{Y}_t (s)|^2 ds+\int_t^T |\mathcal{Z}_t (s)|^2 ds\right]
\leq C \mathbb{E}\left[ |\varphi (t)|^2
+ \left(\int_t^T  |g (t,s)| ds\right)^2
 \right].
\end{aligned}
\end{equation}
For any fixed $t\in [0, T]$, since $\mathcal{Y}_t (r)\in L_{\mathcal{F}_r}^{2} (\Omega; \mathbb{R}^n)$ for all $r\in [t,T]$, by the martingale representation theorem, there exists a unique process $\mathcal{K}_t (\cdot)\in L_{\mathbb{F}}^2 (0,r;  \mathbb{R}^{n\times m})$ such that
\begin{equation*}
\begin{aligned}
\mathcal{Y}_t (r)=\mathbb{E}[\mathcal{Y}_t (r)]+\int_0^r \mathcal{K}_t (s) dW(s),~~t\in [0, T].
\end{aligned}
\end{equation*}
Set
\begin{equation*}
\begin{aligned}
Y(t):=\mathcal{Y}_t (t),~t\in [0,T],~~~~~
Z(t, s):=\left\{\begin{array}{lcl}
\mathcal{Z}_t (s),~~~(t,s)\in \Delta,\\
\mathcal{K}_t (s),~~~(t,s)\in \Delta^c.
\end{array}
\right.
\end{aligned}
\end{equation*}
Thus, this implies that
$\left(Y(\cdot), Z (\cdot,\cdot) \right)\in L_{\mathbb{F}}^{2} (0,T; \mathbb{R}^n)\times L^2 (0,T; L_{\mathbb{F}}^2 (0,T; \mathbb{R}^{n\times m}))$ is an adapted
M-solution to (\ref{58}).

Under (\ref{6}), one can deduce that
\begin{equation*}
\begin{aligned}
\mathbb{E}\left[|Y (t)|^2 +\int_t^T |Z (t, s)|^2 ds\right]
\leq C \mathbb{E}\left[ |\varphi (t)|^2
+ \left(\int_t^T  |g (t,s)| ds\right)^2
 \right].
\end{aligned}
\end{equation*}
Therefore, we have the following:
\begin{equation*}
\begin{aligned}
\mathbb{E}\left[\int_0^T  |Y (t)|^2 dt+\int_0^T \int_t^T |Z (t,s)|^2 dsdt\right]
\leq C \mathbb{E}\left[\int_0^{T} |\varphi (t)|^2 dt
+\int_0^T \left(\int_t^T  |g (t,s)| ds\right)^2 dt
 \right].
\end{aligned}
\end{equation*}
Moreover, from (\ref{4}), it is easy to prove the uniqueness of (\ref{58}).
Obviously, similar to the proof of (\ref{4}), we can get the stability estimate (\ref{5}).
\end{proof}

Next, let us consider the well-posedness of ABSVIE (\ref{41}), where
the pair of processes $(Y(\cdot), Z(\cdot,\cdot))$
satisfies the following equation:
\begin{align}\label{1}
\left\{\begin{array}{lcl} Y(t) = \varphi (t)+\int_{t}^{T} g\left(  \Lambda (t,s) \right)ds-\int_{t}^{T} Z(t, s) dW(s),\ \ \
  t\in [0, T],\\
Y(t) = \varphi (t), \ \ t\in [T, T+K],\\
Z(t, s) = \eta(t, s), \ \ (t, s)\in [0, T+K]^2 \backslash [0, T]^2,
\end{array}
\right.
\end{align}
where
\begin{equation*}
\begin{aligned}
\Lambda (t,s)&:= \left( t, s, Y(s), Z(t, s), Z(s, t), Y(s+\delta (s)),
Z(t, s+\zeta(s)), Z(s+\zeta(s), t),\right.\\
&~~~ \left. \int_{s}^{s+\delta (s)} e^{\lambda (s-\theta)} Y(\theta)d\theta,
\int_{s}^{s+\zeta (s)} e^{\lambda (s-\theta)} Z(t, \theta)d\theta,  \int_{s}^{s+\zeta (s)} e^{\lambda (s-\theta)} Z(\theta, t)d\theta  \right),
\end{aligned}
\end{equation*}
$T>0$ is a finite time horizon, $K\geq 0$ is a constant, $\lambda\in \mathbb{R}$ is the averaging parameter, and
$\delta(\cdot)$ and $\zeta(\cdot)$ are deterministic $\mathbb{R}^{+}$-valued continuous functions defined on $[0,T]$ such that:
\indent $(a)$ For any $s\in [0, T]$, we have
\begin{equation*}
\begin{aligned}
s+\delta (s)\leq T+K,~~~
s+\zeta (s)\leq T+K.
\end{aligned}
\end{equation*}
\indent $(b)$ There exists a constant $M\geq 0$ such that for any nonnegative and
Lebesgue integrable functions $f_1 (\cdot)$ and $f_2 (\cdot,\cdot)$, and for any $t\in [0, T]$, we have
\begin{equation*}
\begin{aligned}
&\int_t^T f_1 (s+\delta (s)) ds \leq M \int_t^{T+K} f_1 (s) ds,~~~~~~~~~
\int_t^T f_2 (t, s+\zeta (s)) ds \leq M \int_t^{T+K} f_2 (t, s) ds,\\
&\int_t^T f_2 (s+\zeta (s), t) ds \leq M \int_t^{T+K} f_2 (s, t) ds.
\end{aligned}
\end{equation*}
In the sequel, we impose the following assumptions to the generator $g$ of ABSVIE (\ref{1}):
\begin{itemize}
\item[(A1)] $g:\Omega\times\Delta\times (\mathbb{R}^n \times\mathbb{R}^{n\times m}\times\mathbb{R}^{n\times m})^3 \mapsto \mathbb{R}^n$
is $\mathcal{F}_T\otimes\mathcal{B}(\Delta\times (\mathbb{R}^n \times\mathbb{R}^{n\times m}\times\mathbb{R}^{n\times m})^3)$-measurable
such that $s\mapsto g(t,s,y,z,\xi,\alpha,\beta,\gamma,\mu,\nu, \psi)$ is $\mathbb{F}$-progressively measurable for all
$(t,y,z,\xi,\alpha,\beta,\gamma,\mu,\nu, \psi)\in [0,T]\times(\mathbb{R}^n \times\mathbb{R}^{n\times m}\times\mathbb{R}^{n\times m})^3$.	
For any $(t,s)\in \Delta$, $g(\cdot,t,s,\cdot,\cdot,\cdot,\cdot,\cdot,\cdot,\cdot,\cdot,\cdot)$ is a map from $\Omega\times \mathbb{R}^n \times\mathbb{R}^{n\times m}\times\mathbb{R}^{n\times m}
\times L_{\mathcal{F}_{r_1}}^2 (\Omega; \mathbb{R}^n)\times L_{\mathcal{F}_{r_2}}^2 (\Omega; \mathbb{R}^{n\times m})
\times L_{\mathcal{F}_{r_3}}^2 (\Omega; \mathbb{R}^{n\times m})
\times L_{\mathcal{F}_{r_4}}^2 (\Omega; \mathbb{R}^n)
\times L_{\mathcal{F}_{r_5}}^2 (\Omega; \mathbb{R}^{n\times m})\times L_{\mathcal{F}_{r_6}}^2 (\Omega; \mathbb{R}^{n\times m})$
to $L_{\mathcal{F}_{s}}^2 (\Omega;\mathbb{R}^n)$
defined by
\begin{equation*}
\begin{aligned}
&~~~~g(\cdot,t,s,\cdot,\cdot,\cdot,\cdot,\cdot,\cdot,\cdot,\cdot,\cdot)(\omega,y,z,\xi,
\alpha,\beta,\gamma,\mu,\nu, \psi)\\
&:=g(\omega,t,s,y,z,\xi,\alpha (\omega),\beta(\omega),\gamma(\omega),\mu(\omega),\nu(\omega), \psi(\omega)),\\
&\forall (\omega,y,z,\xi,
\alpha,\beta,\gamma,\mu,\nu, \psi)\in \Omega\times \mathbb{R}^n \times\mathbb{R}^{n\times m}\times\mathbb{R}^{n\times m}
\times L_{\mathcal{F}_{r_1}}^2 (\Omega; \mathbb{R}^n)\times L_{\mathcal{F}_{r_2}}^2 (\Omega; \mathbb{R}^{n\times m})
\\
&\times L_{\mathcal{F}_{r_3}}^2 (\Omega; \mathbb{R}^{n\times m})\times L_{\mathcal{F}_{r_4}}^2 (\Omega; \mathbb{R}^n)
\times L_{\mathcal{F}_{r_5}}^2 (\Omega; \mathbb{R}^{n\times m})\times L_{\mathcal{F}_{r_6}}^2 (\Omega; \mathbb{R}^{n\times m}),
\end{aligned}
\end{equation*}
where $r_i \in [s, T+K]$, $i=1,\cdots,6$.
Moreover,
\begin{equation}\label{46}
\begin{aligned}
\mathbb{E} \left[ \int_0^T \left(\int_t^T |g_0 (t,s)| ds \right)^2 dt \right]<\infty,
\end{aligned}
\end{equation}
where $g_0 (t,s):=g (t,s,0,0,0,0,0,0,0,0,0)$.
\item[(A2)]
There exists a constant $C\geq 0$ such that for any $(t,s)\in\Delta$, $y_1, y_2 \in \mathbb{R}^n$,
$z_1, z_2, \xi_1, \xi_2 \in \mathbb{R}^{n\times m}$, $\alpha_1 (\cdot), \alpha_2 (\cdot), \mu_1 (\cdot), \mu_2 (\cdot)\in L_{\mathbb{F}}^2 (s, T+K; \mathbb{R}^n)$, $\beta_1 (t, \cdot), \beta_2 (t, \cdot),$ $\gamma_1 (\cdot, t), \gamma_2 (\cdot, t), \nu_1 (t, \cdot)$, $\nu_2 (t, \cdot), \psi_1 (\cdot, t), \psi_2 (\cdot, t)\in L_{\mathbb{F}}^2 (s, T+K; \mathbb{R}^{n\times m})$, we have
\begin{align*}
&~~~|g(t,s,y_1,z_1,\xi_1,\alpha_1 (r_1),\beta_1 (t, r_2),\gamma_1 (r_3, t),\mu_1 (r_4),\nu_1 (t, r_5), \psi_1 (r_6, t))\\
&-g(t,s,y_2,z_2,\xi_2,\alpha_2 (r_1),\beta_2 (t, r_2),\gamma_2 (r_3, t),\mu_2 (r_4),\nu_2 (t, r_5), \psi_2 (r_6, t))|\\
&\leq C\left(|y_1-y_2|+|z_1-z_2|+|\xi_1-\xi_2|+\mathbb{E}\left[\left(|\alpha_1 (r_1)-\alpha_2 (r_1)|+|\beta_1 (t, r_2)-\beta_2 (t, r_2)|\right.\right.\right.\\
&~~~+|\gamma_1(r_3, t)-\gamma_2(r_3, t)|+|\mu_1(r_4)-\mu_2(r_4)|+|\nu_1 (t, r_5)-\nu_2 (t, r_5)|\\
&~~\left.\left.\left.+|\psi_1 (r_6, t)-\psi_2 (r_6, t)|\right)\big| \mathcal{F}_s \right]\right), a.s.
\end{align*}
\end{itemize}

\begin{definition}\label{def1}
A pair of processes $(Y(\cdot),Z(\cdot,\cdot))\in L_{\mathbb{F}}^{2} (0,T+K; \mathbb{R}^n)\times L^2 (0,T+K; L_{\mathbb{F}}^2 (0,T+K; \mathbb{R}^{n\times m}))$ is called an adapted M-solution of ABSVIE (\ref{1}) if (\ref{1}) is satisfied in the usual It\^{o}'s sense for almost everywhere all $t\in [0, T+K]$, and for any $R\in [0,T+K]$, the following relation holds:
\begin{equation}\label{2}
\begin{aligned}
Y(t)=\mathbb{E}[Y(t)|\mathcal{F}_R]+\int_R^t Z(t,s) dW(s),~~~a.e.~t\in [R,T+K].
\end{aligned}
\end{equation}
\end{definition}

For simplicity, denote by $\mathcal{H}^2 [0,T+K]$ the Banach space
\begin{equation*}
\begin{aligned}
\mathcal{H}^2 [0,T+K]=L_{\mathbb{F}}^{2} (0,T+K; \mathbb{R}^n)\times L^2 (0,T+K; L_{\mathbb{F}}^2 (0,T+K; \mathbb{R}^{n\times m}))
\end{aligned}
\end{equation*}
endowed with the norm
\begin{equation*}
\begin{aligned}
\|(Y(\cdot),Z(\cdot,\cdot))\|_{\mathcal{H}^2 [0,T+K]}^2:=
\mathbb{E}\left[\int_0^{T+K} |Y(t)|^2 dt+\int_0^{T+K} \int_0^{T+K}  |Z(t,s)|^2 dsdt\right].
\end{aligned}
\end{equation*}
Let $\mathcal{M}^2 [0,T+K]$ denote the family of all pairs $(Y(\cdot),Z(\cdot,\cdot))\in
\mathcal{H}^2 [0,T+K]$ such that (\ref{2}) holds.
For any $0\leq A<B \leq T+K$, the Banach space $\mathcal{M}^2 [A,B]$ can be defined similarly.

The well-posedness of the adapted M-solution to ABSVIE (\ref{1}) as the main result
of this section is presented as follows.

\begin{theorem}\label{th1}
Under Assumptions $(A1)$-$(A2)$, for any given $(\varphi (\cdot), \eta(\cdot, \cdot))\in \mathcal{M}^2 [0,T+K]$,
ABSVIE (\ref{1}) admits a unique adapted M-solution $(Y(\cdot), Z(\cdot,\cdot))\in \mathcal{M}^2 [0,T+K]$.
In addition, there exists a constant $C>0$ such that
\begin{equation}\label{7}
\begin{aligned}
&~~~\|(Y(\cdot),Z(\cdot,\cdot))\|_{\mathcal{M}^2 [0,T+K]}^2
:= \mathbb{E}\left[\int_0^{T+K} |Y(t)|^2 dt+\int_0^{T+K} \int_0^{T+K}  |Z(t,s)|^2 dsdt\right]\\
&\leq C \mathbb{E}\left[\int_0^{T+K} |\varphi(t)|^2 dt+\int_0^T \left(\int_t^T  |g_0 (t,s)| ds\right)^2 dt +\int_0^{T+K} \int_T^{T+K} |\eta (t,s)|^2 dsdt \right].
\end{aligned}
\end{equation}
For $i=1,2$, suppose that $g_i$ satisfies Assumptions $(A1)$-$(A2)$ and $(\varphi_i (\cdot), \eta_i (\cdot, \cdot))\in \mathcal{M}^2 [0,T+K]$.
Let $(Y_i (\cdot), Z_i (\cdot,\cdot))\in \mathcal{M}^2 [0,T+K]$ be the corresponding adapted M-solution of ABSVIE (\ref{1}) with respect
to $g_i$ and $(\varphi_i (\cdot), \eta_i (\cdot, \cdot))$, we have the following stability estimate:
\begin{equation}\label{8}
\begin{aligned}
&~~~\mathbb{E}\left[\int_0^{T+K} |Y_1 (t)-Y_2 (t)|^2 dt+\int_0^{T+K} \int_0^{T+K}  |Z_1 (t,s)-Z_2 (t,s)|^2 dsdt\right] \\
&\leq C \mathbb{E}\left[\int_0^{T+K} |\varphi_1 (t)-\varphi_2 (t)|^2 dt
+\int_0^{T+K} \int_T^{T+K}  |\eta_1 (t,s)-\eta_2 (t,s)|^2 dsdt \right.\\
&~~~\left.+\int_0^T \left(\int_t^T  | \Delta g (t,s)| ds\right)^2 dt\right],
\end{aligned}
\end{equation}
where
\begin{equation*}
\begin{aligned}
\Delta g(t,s)&=g_1 \left( t, s, Y_1(s), Z_1(t, s), Z_1(s, t), Y_1(s+\delta (s)),
Z_1(t, s+\zeta(s)), Z_1(s+\zeta(s), t),\right.\\
&~~~\left.\int_{s}^{s+\delta (s)} e^{\lambda (s-\theta)} Y_1(\theta)d\theta,
\int_{s}^{s+\zeta (s)} e^{\lambda (s-\theta)} Z_1(t, \theta)d\theta,  \int_{s}^{s+\zeta (s)} e^{\lambda (s-\theta)} Z_1(\theta, t)d\theta  \right)\\
&~~~-g_2 \left( t, s, Y_1(s), Z_1(t, s), Z_1(s, t), Y_1(s+\delta (s)),
Z_1(t, s+\zeta(s)), Z_1(s+\zeta(s), t),\right.\\
&~~~\left.\int_{s}^{s+\delta (s)} e^{\lambda (s-\theta)} Y_1(\theta)d\theta,
\int_{s}^{s+\zeta (s)} e^{\lambda (s-\theta)} Z_1(t, \theta)d\theta,  \int_{s}^{s+\zeta (s)} e^{\lambda (s-\theta)} Z_1(\theta, t)d\theta  \right).
\end{aligned}
\end{equation*}
\end{theorem}
\begin{proof}
The proof is rather technical and thus we divide the whole proof into the following
four steps, as shown in Figure $1$.
In fact, for any given $\left(\varphi(\cdot), \eta (\cdot,\cdot) \right)\in \mathcal{M}^2 [0,T+K]$, we already know that $Y(t)=\varphi (t)$ for any $t\in [T, T+K]$, and $Z(t,s)=\eta (t,s)$ for any $(t,s)\in [0, T+K]^2 \backslash [0, T]^2$. Therefore,
we only need to prove that ABSVIE (\ref{1}) admits a unique adapted M-solution on $[0, T]$.
\begin{figure}
\centering
\includegraphics[height=5cm,width=5cm]{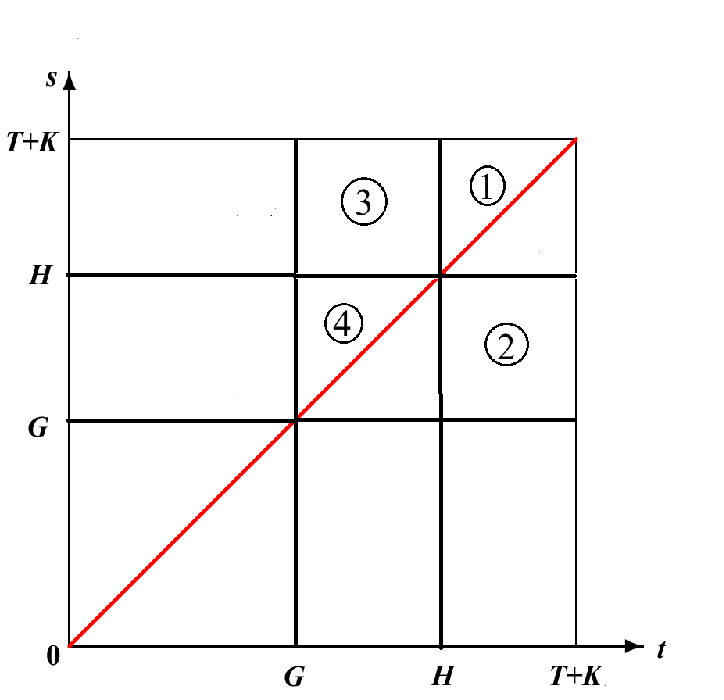}\\
{\normalsize{\textbf{Figure 1}~ The domain for $(Y(\cdot), Z(\cdot,\cdot))$}}
\label{fig.1}
\end{figure}

Step 1: We shall prove that ABSVIE (\ref{1}) admits a unique adapted M-solution
on $[H,T+K]$
for some $H\in [0,T)$. For this, we introduce an equivalent norm for
$\mathcal{M}^2 [H,T+K]$ as follows:
\begin{equation*}
\begin{aligned}
\|(Y(\cdot),Z(\cdot,\cdot))\|_{\mathcal{M}^2 [H,T+K]}^2:=
\mathbb{E}\left[\int_H^{T+K} |Y(t)|^2 dt+\int_H^{T+K} \int_t^{T+K}  |Z(t,s)|^2 dsdt\right].
\end{aligned}
\end{equation*}
For any given $(\varphi (\cdot), \eta(\cdot, \cdot))$,
$(y(\cdot),z(\cdot,\cdot))\in \mathcal{M}^2 [H,T+K]$,
we consider the following equation on $[H, T+K]$:
\begin{align}\label{9}
\left\{\begin{array}{lcl} Y(t) = \varphi (t)+\int_{t}^{T} g\left( t, s, y(s), z(t, s), z(s, t), y(s+\delta (s)),
z(t, s+\zeta(s)), z(s+\zeta(s), t), \right.\\
\ \ \ \ \ \ \  \  \  \left.  \int_{s}^{s+\delta (s)} e^{\lambda (s-\theta)} y(\theta)d\theta,
\int_{s}^{s+\zeta (s)} e^{\lambda (s-\theta)} z(t, \theta)d\theta,  \int_{s}^{s+\zeta (s)} e^{\lambda (s-\theta)} z(\theta, t)d\theta  \right)ds\\
\ \ \ \ \ \ \  \  \ -\int_{t}^{T} Z(t, s) dW(s),\ \ \
  t\in [H, T],\\
Y(t) = \varphi (t), \ \ t\in [T, T+K],\\
Z(t, s) = \eta(t, s), \ \ (t, s)\in [H, T+K]^2 \backslash [H, T]^2.
\end{array}
\right.
\end{align}
From Lemma \ref{lem1}, it is obvious to see that BSVIE (\ref{9}) admits a unique adapted M-solution $(Y(\cdot), Z(\cdot, \cdot))\in \mathcal{M}^2 [H,T]$. By the fact that $Y(t)=\varphi (t)$ for any $t\in [T, T+K]$, and $Z(t,s)=\eta (t,s)$ for any $(t,s)\in [H, T+K]^2 \backslash [H, T]^2$,
where $\left(\varphi(\cdot), \eta (\cdot,\cdot) \right)\in \mathcal{M}^2 [H,T+K]$, we know that $(Y(\cdot), Z(\cdot, \cdot))\in \mathcal{M}^2 [H,T+K]$ is the unique adapted
M-solution of (\ref{9}).
If $t\in[H,T]$, then it follows from (\ref{4}) that
\begin{equation*}
\begin{aligned}
&~~~\mathbb{E}\left[\int_H^{T} |Y(t)|^2 dt+\int_H^{T} \int_t^{T}  |Z(t,s)|^2 dsdt\right]\\
&\leq C \mathbb{E}\left[\int_H^{T} |\varphi (t)|^2 dt
+\int_H^T \left(\int_t^T  \left|g\left( t, s, y(s), z(t,s), z(s, t), y(s+\delta (s)),
z(t,s+\zeta(s)),  \right.\right.\right.\right.\\
&~~\left.\left.\left. \left. z(s+\zeta(s), t), \int_{s}^{s+\delta (s)} e^{\lambda (s-\theta)} y(\theta)d\theta,
\int_{s}^{s+\zeta (s)} e^{\lambda (s-\theta)} z(t,\theta)d\theta,  \int_{s}^{s+\zeta (s)} e^{\lambda (s-\theta)} z(\theta, t)d\theta  \right)\right| ds\right)^2 dt
 \right]\\
 &\leq C \mathbb{E}\left\{\int_H^{T} |\varphi (t)|^2 dt
+\int_H^T \left(\int_t^T |g_0 (t,s)| ds \right)^2 dt+\int_H^T \left[\int_t^T  \left( |y(s)|+|z(t,s)|+|z(s, t)| \right.\right.\right.\\
&~~~+\mathbb{E}[|y(s+\delta (s))| \big| \mathcal{F}_s]
+\mathbb{E}[|z(t,s+\zeta(s))| \big| \mathcal{F}_s] +\mathbb{E}[|z(s+\zeta(s), t)| \big| \mathcal{F}_s]\\
&~~~
+\mathbb{E}\left[\int_{s}^{s+\delta (s)} e^{\lambda (s-\theta)} |y(\theta)|d\theta \Bigg|\mathcal{F}_s\right]
+\mathbb{E}\left[\int_{s}^{s+\zeta (s)} e^{\lambda (s-\theta)} |z(t,\theta)|d\theta \Bigg|\mathcal{F}_s\right]\\
&~~\left.\left.\left.
+\mathbb{E}\left[\int_{s}^{s+\zeta (s)} e^{\lambda (s-\theta)} |z(\theta, t)|d\theta \Bigg|\mathcal{F}_s\right] \right)ds\right]^2 dt
\right\}
\end{aligned}
\end{equation*}
\begin{equation*}
\begin{aligned}
&\leq C \mathbb{E}\left\{\int_H^{T} |\varphi (t)|^2 dt
+\int_H^T \left(\int_t^T |g_0 (t,s)| ds \right)^2 dt+\int_H^T \int_t^T  \left[ |y(s)|^2+|z(t,s)|^2+|z(s, t)|^2 \right.\right.\\
&~~~+|y(s+\delta (s))|^2
+|z(t,s+\zeta(s))|^2 +|z(s+\zeta(s), t)|^2 +\left(\int_{s}^{s+\delta (s)} e^{\lambda (s-\theta)} |y(\theta)|d\theta \right)^2 \\
&~~~\left. \left.
+\left(\int_{s}^{s+\zeta (s)} e^{\lambda (s-\theta)} |z(t,\theta)|d\theta  \right)^2
+\left(\int_{s}^{s+\zeta (s)} e^{\lambda (s-\theta)} |z(\theta, t)|d\theta  \right)^2 \right]ds dt
\right\}\\
&\leq C \mathbb{E}\left\{\int_H^{T} |\varphi (t)|^2 dt
+\int_H^T \left(\int_t^T |g_0 (t,s)| ds \right)^2 dt+  \int_H^T |y(t)|^2 dt+\int_H^T \int_t^T |z(t, s)|^2 dsdt \right.\\
&~~~+\int_H^T \int_t^T |z(s, t)|^2 dsdt+M \int_H^{T+K}|y(t)|^2 dt
+M \int_H^{T+K} \int_t^{T+K} |z(t,s)|^2 dsdt\\
&~~~+M \int_H^{T+K} \int_t^{T+K} |z(s, t)|^2 dsdt+\frac{1-e^{-2\lambda(T+K)}}{2\lambda} (T-H)^2 \int_H^{T+K}|y(t)|^2 dt\\
&~~~+\frac{1-e^{-2\lambda(T+K)}}{2\lambda} (T-H) \int_H^{T+K} \int_t^{T+K} |z(t,s)|^2 dsdt\\
&~~~\left.+\frac{1-e^{-2\lambda(T+K)}}{2\lambda} (T-H) \int_H^{T+K} \int_t^{T+K} |z(s, t)|^2 dsdt\right\}\\
&\leq C \mathbb{E}\left\{\int_H^{T} |\varphi (t)|^2 dt
+\int_H^T \left(\int_t^T |g_0 (t,s)| ds \right)^2 dt+\int_H^{T+K}|y(t)|^2 dt dsdt  \right.\\
&~~~\left. +\int_H^{T+K} \int_t^{T+K} |z(t, s)|^2 dsdt
+\int_H^{T+K} \int_H^t |z(t, s)|^2 dsdt\right\}\\
&\leq C \mathbb{E}\left\{\int_H^{T} |\varphi (t)|^2 dt+\int_H^T \left(\int_t^T |g_0 (t,s)| ds \right)^2 dt
+\int_H^{T+K}|y(t)|^2 dt
 \right.\\
&~~~\left. +\int_H^{T+K} \int_t^{T+K} |z (t,s)|^2 dsdt\right\}.
\end{aligned}
\end{equation*}
Recall that $Y(t)=\varphi (t)$ for any $t\in [T, T+K]$, and $Z(t,s)=\eta (t,s)$ for any $(t,s)\in [H, T+K]^2 \backslash [H, T]^2$,
where $\left(\varphi(\cdot), \eta (\cdot,\cdot) \right)\in \mathcal{M}^2 [H,T+K]$, we derive that
\begin{equation}\label{61}
\begin{aligned}
&~~~\|(Y(\cdot),Z(\cdot,\cdot))\|_{\mathcal{M}^2 [H,T+K]}^2\\
&\leq C \left\{\|(y(\cdot),z(\cdot,\cdot))\|_{\mathcal{M}^2 [H,T+K]}^2+\mathbb{E}\left[\int_H^{T+K} |\varphi(t)|^2 dt
+\int_H^{T+K} \int_T^{T+K} |\eta (t,s)|^2 dsdt\right.\right.\\
&~~~\left.\left.+\int_H^T \left(\int_t^T |g_0 (t,s)| ds \right)^2 dt\right]\right\}<\infty.
\end{aligned}
\end{equation}
Thus, let us define a map $\Pi$ from $\mathcal{M}^2 [H,T+K]$ to itself by
\begin{equation}\label{47}
\begin{aligned}
\Pi (y(\cdot),z(\cdot,\cdot))=(Y(\cdot),Z(\cdot,\cdot)),~\mbox{for any}~(y(\cdot),z(\cdot,\cdot))\in \mathcal{M}^2 [H,T+K].
\end{aligned}
\end{equation}
Subsequently, we shall further check that $\Pi$ is a contractive map when $T-H>0$ is small enough.
For any $(y_i(\cdot), z_i(\cdot,\cdot))\in \mathcal{M}^2 [H,T+K]$, $i=1,2$,
we denote by $(Y_i(\cdot), Z_i(\cdot,\cdot))=\Pi(y_i(\cdot), z_i(\cdot,\cdot))$. Let
\begin{align*}
(\Delta y(\cdot), \Delta z(\cdot,\cdot))&=(y_1(\cdot)-y_2(\cdot), z_1(\cdot,\cdot)-z_2(\cdot,\cdot)),\\
(\Delta Y(\cdot), \Delta Z(\cdot,\cdot))&=(Y_1(\cdot)-Y_2(\cdot), Z_1(\cdot,\cdot)-Z_2(\cdot,\cdot)).
\end{align*}
According to the stability estimate (\ref{5}) and the Lipschitz conditions of $g$, we can show that
\begin{equation*}
\begin{aligned}
&~~~~\mathbb{E}\left[\int_H^T |\Delta Y(t)|^2 dt+\int_H^T \int_t^T |\Delta Z(t,s)|^2 dsdt\right]\\
&\leq C \mathbb{E}\left[\int_H^T \left(\int_t^T
\left|g\left(t, s, y_1(s), z_1(t, s), z_1(s, t), y_1(s+\delta (s)),
z_1(t, s+\zeta(s)), z_1(s+\zeta(s), t), \right.\right.\right.\right.\\
& ~~~\left.  \int_{s}^{s+\delta (s)} e^{\lambda (s-\theta)} y_1(\theta)d\theta,
\int_{s}^{s+\zeta (s)} e^{\lambda (s-\theta)} z_1(t, \theta)d\theta,  \int_{s}^{s+\zeta (s)} e^{\lambda (s-\theta)} z_1(\theta, t)d\theta\right)\\
&~~~-g\left(t, s, y_2(s), z_2(t, s), z_2(s, t), y_2(s+\delta (s)),
z_2(t, s+\zeta(s)), z_2(s+\zeta(s), t), \right.\\
&~~~ \left.\left.\left.\left.  \int_{s}^{s+\delta (s)} e^{\lambda (s-\theta)} y_2(\theta)d\theta,
\int_{s}^{s+\zeta (s)} e^{\lambda (s-\theta)} z_2(t, \theta)d\theta,  \int_{s}^{s+\zeta (s)} e^{\lambda (s-\theta)} z_2(\theta, t)d\theta\right)\right| ds\right)^2 dt\right]\\
&\leq C \mathbb{E}\left\{
\int_H^T \left(\int_t^T |\Delta y(s)|ds \right)^2 dt+\int_H^T \left(\int_t^T |\Delta z(t, s)|ds\right)^2 dt
+\int_H^T \left(\int_t^T |\Delta z(s, t)|ds\right)^2 dt
\right.\\
&~~~+\int_H^T \left(\int_t^T|\Delta y(s+\delta (s))|ds \right)^2 dt
+\int_H^T \left(\int_t^T |\Delta z(t,s+\zeta(s))|ds \right)^2 dt\\
&~~~+\int_H^T \left(\int_t^T |\Delta z(s+\zeta(s), t)|ds \right)^2 dt
+\int_H^T \left(\int_t^T \left|\int_{s}^{s+\delta (s)} e^{\lambda (s-\theta)} \Delta y(\theta) d\theta \right|ds\right)^2 dt\\
&~~~\left.+\int_H^T \left(\int_t^T \left|\int_{s}^{s+\zeta (s)} e^{\lambda (s-\theta)} \Delta z(t,\theta) d\theta  \right|ds\right)^2 dt
+\int_H^T \left(\int_t^T \left|\int_{s}^{s+\zeta (s)} e^{\lambda (s-\theta)} \Delta z(\theta, t) d\theta  \right|ds\right)^2 dt
\right\}\\
&\leq C(T-H) \mathbb{E}\left\{
\int_H^T  |\Delta y(t)|^2 dt+\int_H^T \int_t^T |\Delta z(t, s)|^2 ds dt
+\int_H^T \int_t^T |\Delta z(s, t)|^2 ds dt
\right.\\
&~~~+\int_H^T \int_t^T|\Delta y(s+\delta (s))|^2 ds  dt+\int_H^T \int_t^T |\Delta z(t,s+\zeta(s))|^2 ds  dt
+\int_H^T \int_t^T |\Delta z(s+\zeta(s), t)|^2 ds  dt\\
&~~~+\int_H^T \int_t^T \left|\int_{s}^{s+\delta (s)} e^{\lambda (s-\theta)} \Delta y(\theta) d\theta \right|^2 ds dt
+\int_H^T \int_t^T \left|\int_{s}^{s+\zeta (s)} e^{\lambda (s-\theta)} \Delta z(t,\theta) d\theta \right|^2 ds dt\\
&~~~\left.
+\int_H^T \int_t^T \left|\int_{s}^{s+\zeta (s)} e^{\lambda (s-\theta)} \Delta z(\theta, t) d\theta \right|^2 ds dt
\right\}
\end{aligned}
\end{equation*}
\begin{equation*}
\begin{aligned}
&\leq C(T-H) \mathbb{E}\left\{
\int_H^T  |\Delta y(t)|^2 dt+\int_H^T \int_t^T |\Delta z(t,s)|^2 ds  dt
+\int_H^T \int_H^t |\Delta z(t, s)|^2 ds dt
\right.\\
&~~~+M(T-H)\int_H^{T+K} |\Delta y(t)|^2 dt
+M\int_H^{T+K} \int_t^{T+K} |\Delta z(t,s)|^2 ds  dt
+M\int_H^{T+K} \int_H^t |\Delta z(t, s)|^2 ds  dt
\\
&~~~+\int_H^T \int_t^T \left(\int_{s}^{s+\delta (s)} e^{2\lambda (s-\theta)}d\theta \int_{s}^{s+\delta (s)}  |\Delta y(\theta)|^2 d\theta \right) ds dt\\
&~~~+\int_H^T \int_t^T \left(\int_{s}^{s+\zeta (s)} e^{2\lambda (s-\theta)}d\theta \int_{s}^{s+\zeta (s)}  |\Delta z(t,\theta)|^2 d\theta \right) ds dt\\
&~~~\left.+\int_H^T \int_t^T \left(\int_{s}^{s+\zeta (s)} e^{2\lambda (s-\theta)}d\theta \int_{s}^{s+\zeta (s)}  |\Delta z(\theta, t)|^2 d\theta \right) ds dt
\right\}\\
&\leq C(T-H)\mathbb{E} \left\{\left[ 2+M+\left(M+\frac{1-e^{-2\lambda(T+K)}}{2\lambda}\right)(T-H)+\frac{1-e^{-2\lambda(T+K)}}{2\lambda} (T-H)^2 \right]
\right.\\
&~~~\left. \cdot\int_H^{T+K}  |\Delta y(t)|^2 dt+
\left[1+M+\frac{1-e^{-2\lambda(T+K)}}{2\lambda} (T-H)\right] \int_H^{T+K} \int_t^{T+K}  |\Delta z(t,s)|^2 ds dt
\right\}\\
&\leq CI(T-H)\mathbb{E}\left[\int_H^{T+K}  |\Delta y(t)|^2 dt+\int_H^{T+K}\int_t^{T+K} |\Delta z(t,s)|^2 dsdt \right],
\end{aligned}
\end{equation*}
where
\begin{equation*}
\begin{aligned}
I=2+M+\left(M+\frac{1-e^{-2\lambda(T+K)}}{2\lambda}\right)(T-H)+\frac{1-e^{-2\lambda(T+K)}}{2\lambda} (T-H)^2.
\end{aligned}
\end{equation*}
By the fact that $\Delta Y(t)=0$ for any $t\in [T, T+K]$, and $\Delta Z(t,s)=0$ for any $(t,s)\in [H, T+K]^2 \backslash [H, T]^2$, it is easy to see that
\begin{equation*}
\begin{aligned}
&~~~~\mathbb{E}\left[\int_H^{T+K} |\Delta Y(t)|^2 dt+\int_H^{T+K} \int_t^{T+K} |\Delta Z(t,s)|^2 dsdt\right]\\
&\leq CI(T-H)\mathbb{E}\left[\int_H^{T+K}  |\Delta y(t)|^2 dt
+\int_H^{T+K}\int_t^{T+K} |\Delta z(t,s)|^2 dsdt\right].
\end{aligned}
\end{equation*}
Therefore, this implies that
\begin{equation*}
\begin{aligned}
&~~~\left\|\Pi(y_1(\cdot), z_1(\cdot,\cdot))-\Pi(y_2(\cdot), z_2(\cdot,\cdot))\right\|_{\mathcal{M}^2 [H,T+K]}^2\\
&\leq CI(T-H)\left\|(y_1(\cdot), z_1(\cdot,\cdot))-(y_2(\cdot), z_2(\cdot,\cdot))\right\|_{\mathcal{M}^2 [H,T+K]}^2.
\end{aligned}
\end{equation*}
If $T-H>0$ is small enough such that $CI(T-H)<1$, then one can see that $\Pi$ is a contractive map on the Banach space $\mathcal{M}^2 [H,T+K]$.
Thus, from Banach contraction mapping principle, the map $\Pi$ admits a unique fixed point $(Y(\cdot), Z(\cdot,\cdot))$ such that $\Pi(Y(\cdot), Z(\cdot,\cdot))=(Y(\cdot), Z(\cdot,\cdot))$. It immediately follows that $(Y(\cdot), Z(\cdot,\cdot))\in \mathcal{M}^2 [H,T+K]$ is the unique adapted M-solution of ABSVIE (\ref{1}) on $[H,T+K]$.
By (\ref{61}), we have the following estimate:
\begin{equation}\label{10}
\begin{aligned}
&~~~\|(Y(\cdot),Z(\cdot,\cdot))\|_{\mathcal{M}^2 [H,T+K]}^2
:=\mathbb{E}\left[\int_H^{T+K} |Y(t)|^2 dt+\int_H^{T+K} \int_t^{T+K}  |Z(t,s)|^2 dsdt\right] \\
&\leq C \mathbb{E}\left[\int_H^{T+K} |\varphi(t)|^2 dt+\int_H^T \left(\int_t^T  |g_0 (t,s)| ds\right)^2 dt +\int_H^{T+K} \int_T^{T+K} |\eta (t,s)|^2 dsdt \right].
\end{aligned}
\end{equation}

Step 2: Let $G \in [0, H)$ be a constant to be determined later. Next, we continue to confirm the values of $Z(t,s)$ for $(t,s)\in [H, T+K]\times [G, H]$. From Step 1, we know that $\mathbb{E}[Y(t)|\mathcal{F}_H]\in L_{\mathcal{F}_H}^{2} (\Omega; \mathbb{R}^n)$ for each $t\in [H,T]$.
Then by the martingale representation theorem, there exists a unique process $Z(t,\cdot)\in L_{\mathbb{F}}^2 (G,H;  \mathbb{R}^{n\times m})$ such that
\begin{equation*}
\begin{aligned}
\mathbb{E}\left[Y(t)|\mathcal{F}_H\right]=\mathbb{E}\left[Y(t)|\mathcal{F}_G\right]+\int_G^H Z(t,s) dW(s),~~t\in[H,T].
\end{aligned}
\end{equation*}
Since $Y(t)=\varphi (t)$ for any $t\in [T, T+K]$, and $Z(t,s)=\eta (t,s)$ for any $(t,s)\in [T, T+K] \times [G, H]$, where $\left(\varphi(\cdot), \eta (\cdot,\cdot) \right)\in \mathcal{M}^2 [G,T+K]$, we have the following:
\begin{equation*}
\begin{aligned}
\mathbb{E}\left[Y(t)|\mathcal{F}_H\right]=\mathbb{E}\left[Y(t)|\mathcal{F}_G\right]+\int_G^H Z(t,s) dW(s),~~t\in[H,T+K].
\end{aligned}
\end{equation*}
Thus,
\begin{equation*}
\begin{aligned}
\mathbb{E}\left[\int_G^H |Z(t,s)|^2 ds\right]\leq \mathbb{E}\left[|Y(t)|^2 \right],~~t\in[H, T+K].
\end{aligned}
\end{equation*}
Then (\ref{10}) yields that
\begin{equation}\label{11}
\begin{aligned}
&~~~\mathbb{E}\left[\int_H^{T+K}\int_G^H |Z(t,s)|^2 dsdt\right]\leq \mathbb{E}\left[\int_H^{T+K}|Y(t)|^2 dt\right]\\
&\leq C \mathbb{E}\left[\int_H^{T+K} |\varphi(t)|^2 dt+\int_H^T \left(\int_t^T  |g_0 (t,s)| ds\right)^2 dt +\int_H^{T+K} \int_T^{T+K} |\eta (t,s)|^2 dsdt \right].
\end{aligned}
\end{equation}
So far, the values of $(Y(t),Z(t,s))$ have been determined for $(t,s)\in [H, T+K]\times [G, T+K]$.
Under (\ref{10}) and (\ref{11}), we get
\begin{equation}\label{12}
\begin{aligned}
&~~~\mathbb{E}\left[\int_H^{T+K}|Y(t)|^2 dt+\int_H^{T+K}\int_G^{T+K} |Z(t,s)|^2 dsdt\right]\\
&\leq C \mathbb{E}\left[\int_H^{T+K} |\varphi(t)|^2 dt+\int_H^T \left(\int_t^T  |g_0 (t,s)| ds\right)^2 dt +\int_H^{T+K} \int_T^{T+K} |\eta (t,s)|^2 dsdt \right].
\end{aligned}
\end{equation}

Step 3: For $(t,s)\in [G, H]\times [H, T+K]$, it is clear that the values of $(Y(s),Z(s,t))$ are already confirmed
according to Steps $1$ and $2$. In the sequel of this paper, one can continue to determine the values of $Z(t,s)$
for $(t,s)\in [G, H]\times [H, T+K]$. More precisely, we investigate the stochastic Fredholm integral equation of the following form:
\begin{align}\label{13}
\left\{\begin{array}{lcl} \varphi^H (t) = \varphi (t)+\int_{H}^{T} \widehat{g} \left( t, s, Z(t, s),
Z(t, s+\zeta(s)),
\int_{s}^{s+\zeta (s)} e^{\lambda (s-\theta)} Z(t, \theta)d\theta \right)ds\\
\ \ \ \ \ \ \ \ \  \  \ -\int_{H}^{T} Z(t, s) dW(s),\ \ \
  t\in [G, H],\\
Z(t, s) = \eta(t, s), \ \ (t, s)\in [G, H]\times [T, T+K],
\end{array}
\right.
\end{align}
where
\begin{equation*}
\begin{aligned}
&\widehat{g} (t,s,z_1,z_2,z_3):= g \left( t, s, Y(s), z_1, Z(s, t), Y(s+\delta (s)), z_2, Z(s+\zeta(s), t), \right.\\
& ~~~~~~~~~~~~~~~~~~~~~~~~~\left.  \int_{s}^{s+\delta (s)} e^{\lambda (s-\theta)} Y(\theta)d\theta,
z_3,  \int_{s}^{s+\zeta (s)} e^{\lambda (s-\theta)} Z(\theta, t)d\theta  \right),\\
& \forall (t,s,z_1,z_2,z_3)\in [G, H]\times [H, T]\times \mathbb{R}^{n\times m} \times  L_{\mathcal{F}_{r_2}}^2 (\Omega; \mathbb{R}^{n\times m})\times L_{\mathcal{F}_{r_5}}^2 (\Omega; \mathbb{R}^{n\times m}).
\end{aligned}
\end{equation*}
In what follows, let us prove that (\ref{13}) admits a unique adapted solution
$(\varphi^H (\cdot), Z (\cdot,\cdot) ) \in L_{\mathcal{F}_{H}}^{2} (G, H; \mathbb{R}^n) \times L^2 (G,H;
L_{\mathbb{F}}^2 (H, T+K; \mathbb{R}^{n\times m}) )$. Indeed,
for any fixed $t\in [G, H]$, we consider the following ABSDE:
\begin{align}\label{59}
\left\{\begin{array}{lcl} \widehat{\mathcal{Y}}_t (r) = \varphi (t)+\int_{r}^{T} \widehat{g}( t, s, \widehat{\mathcal{Z}}_t (s),
\widehat{\mathcal{Z}}_t (s+\zeta(s)),
\int_{s}^{s+\zeta (s)} e^{\lambda (s-\theta)} \widehat{\mathcal{Z}}_t (\theta)d\theta )ds
-\int_{r}^{T} \widehat{\mathcal{Z}}_t (s) dW(s),\\
\ \ \ \ \ \ \ \ \ \   r\in [H, T],\\
\widehat{\mathcal{Z}}_t (r) = \eta(t, r), \ \ r\in [T, T+K].
\end{array}
\right.
\end{align}
Similar to the proof of Theorem 3.4 in \cite{Meng2015} and the proof of Theorem 4.2 in \cite{Peng2009}, and according to the standard Banach contraction mapping principle, we can obtain that ABSDE (\ref{59}) admits a unique
adapted solution
$(\widehat{\mathcal{Y}}_t (\cdot), \widehat{\mathcal{Z}}_t (\cdot))\in L_{\mathbb{F}}^{2} (H,T; \mathbb{R}^n)\times L_{\mathbb{F}}^2 (H,T+K; \mathbb{R}^{n\times m})$ and the following estimate holds:
\begin{equation}\label{60}
\begin{aligned}
&~~~\mathbb{E}\left[ \int_H^T |\widehat{\mathcal{Y}}_t (s)|^2 ds+\int_H^T |\widehat{\mathcal{Z}}_t (s)|^2 ds\right] \\
&\leq C \mathbb{E}\left[ |\varphi (t)|^2
+ \left(\int_H^T  |\widehat{g} (t,s,0,0,0)| ds\right)^2
+\int_T^{T+K}  |\eta (t,s)|^2 ds \right].
\end{aligned}
\end{equation}
Define
\begin{equation*}
\begin{aligned}
\varphi^H (t):=\widehat{\mathcal{Y}}_t (H),~t\in [G,H],~~~~~
Z(t, s):=
\widehat{\mathcal{Z}}_t (s),~(t,s)\in [G,H]\times [H,T+K].
\end{aligned}
\end{equation*}
Further, for each $t\in [G,H]$, since $\varphi^H (t):=\widehat{\mathcal{Y}}_t (H)\in L_{\mathcal{F}_H}^{2} (\Omega; \mathbb{R}^n)$, it is obvious that (\ref{13}) has a unique adapted solution $(\varphi^H (\cdot), Z (\cdot,\cdot) ) \in L_{\mathcal{F}_{H}}^{2} (G, H; \mathbb{R}^n) \times L^2 (G,H;
L_{\mathbb{F}}^2 (H, T+K; \mathbb{R}^{n\times m}))$.
By (\ref{60}), it holds that
\begin{equation*}
\begin{aligned}
&~~~\mathbb{E}\left[|\varphi^H (t)|^2 +\int_H^{T} |Z(t,s)|^2 ds\right]\\
&\leq C \mathbb{E}\left[ |\varphi(t)|^2 + \left(\int_H^T  \left|g \left( t, s, Y(s),0, Z(s, t), Y(s+\delta (s)), 0, Z(s+\zeta(s), t),  \right.\right.\right.\right.\\
& ~~~\left.\left.\left.\left.  \int_{s}^{s+\delta (s)} e^{\lambda (s-\theta)} Y(\theta)d\theta,
0,  \int_{s}^{s+\zeta (s)} e^{\lambda (s-\theta)} Z(\theta, t)d\theta  \right)\right| ds\right)^2  + \int_T^{T+K} |\eta (t,s)|^2 ds \right]\\
&\leq C \mathbb{E}\left[ |\varphi(t)|^2+ \left(\int_H^T  |g_0 (t,s)| ds\right)^2
+\int_H^{T+K} |Y(t)|^2 dt+ \int_H^{T+K} |Z(s,t)|^2 ds\right.\\
&\left.~~~+ \int_T^{T+K} |\eta (t,s)|^2 ds\right].
\end{aligned}
\end{equation*}
Note that $Z(t, s) = \eta(t, s)$ for any $(t, s)\in [G, H]\times [T, T+K]$, and by the basic estimate (\ref{12}), we know
\begin{equation}\label{14}
\begin{aligned}
&~~~\mathbb{E}\left[ \int_G^H  |\varphi^H (t)|^2 dt+\int_G^{H} \int_H^{T+K} |Z(t,s)|^2 ds dt \right]\\
&\leq C \mathbb{E}\left[\int_G^H |\varphi(t)|^2 dt+\int_G^H \left(\int_H^T  |g_0 (t,s)| ds\right)^2 dt
+\int_H^{T+K} |Y(t)|^2 dt \right.\\
&~~~\left.+\int_G^H \int_H^{T+K} |Z(s,t)|^2 ds dt+ \int_G^H \int_T^{T+K} |\eta (t,s)|^2 ds dt \right]\\
&\leq C \mathbb{E}\left[\int_G^H |\varphi(t)|^2 dt+\int_G^H \left(\int_H^T  |g_0 (t,s)| ds\right)^2 dt
+\int_H^{T+K} |Y(t)|^2 dt \right.\\
&~~~\left.+\int_H^{T+K} \int_G^H |Z(t,s)|^2 dsdt+ \int_G^H \int_T^{T+K} |\eta (t,s)|^2 ds dt \right]\\
&\leq C \mathbb{E}\left[\int_G^{T+K} |\varphi(t)|^2 dt+\int_G^T \left(\int_t^T  |g_0 (t,s)| ds\right)^2 dt +\int_G^{T+K} \int_T^{T+K} |\eta (t,s)|^2 dsdt \right].
\end{aligned}
\end{equation}

Step 4: From Steps $1$ to $3$, it is easy to check that the following processes have been determined:
\begin{equation}\label{15}
\begin{aligned}
\left\{\begin{array}{lcl} Y(t), ~t\in [H,T+K],\\
Z (t,s),~(t,s)\in\left([H,T+K]\times [G, T+K]\right)\bigcup \left([G, H]\times [H, T+K]  \right).
\end{array}
\right.
\end{aligned}
\end{equation}
For $t\in [G, H]$, we introduce the following equation:
\begin{equation}\label{16}
\begin{aligned}
 Y(t) = \varphi^H (t)+\int_{t}^{H} g\left(\Lambda ( t, s)   \right)ds
-\int_{t}^{H} Z(t, s) dW(s),\ \ \
  t\in [G, H].
\end{aligned}
\end{equation}
Since the values of $(Y(\cdot), Z(\cdot,\cdot))$ given in (\ref{15}) are already defined
and $\varphi^H (\cdot)\in L_{\mathcal{F}_{H}}^{2} (G, H$; $\mathbb{R}^n)$, one can have that
(\ref{16}) is an ABSVIE on $[G,H]$.
Using the similar method as in Step 1, we can know that (\ref{16}) admits a unique adapted M-solution $(Y(t), Z(t,s))\in\mathcal{M}^2 [G,H]$.
It then follows from (\ref{10}) and (\ref{14}) that
\begin{equation}\label{17}
\begin{aligned}
&~~~\mathbb{E}\left[ \int_G^H  |Y (t)|^2 dt+\int_G^{H} \int_G^{H} |Z(t,s)|^2 ds dt \right]\\
&\leq C \mathbb{E}\left[\int_G^H |\varphi^H (t)|^2 dt+\int_G^H \left(\int_t^H  |g_0 (t,s)| ds\right)^2 dt
+ \int_G^H \int_T^{T+K} |\eta (t,s)|^2 ds dt \right]
\end{aligned}
\end{equation}
\begin{equation*}
\begin{aligned}
&\leq C \mathbb{E}\left[\int_G^{T+K} |\varphi(t)|^2 dt+\int_G^T \left(\int_t^T  |g_0 (t,s)| ds\right)^2 dt +\int_G^{T+K} \int_T^{T+K} |\eta (t,s)|^2 dsdt \right].
\end{aligned}
\end{equation*}
To sum up, it has been proved the existence and uniqueness of adapted M-solutions to ABSVIE (\ref{1}) on $[G,T+K]$.
By (\ref{12}), (\ref{14}) and (\ref{17}) the following holds:
\begin{equation*}
\begin{aligned}
&~~~\mathbb{E}\left[ \int_G^{T+K}  |Y (t)|^2 dt+\int_G^{T+K} \int_G^{T+K} |Z(t,s)|^2 ds dt \right]\\
&\leq C \mathbb{E}\left[\int_G^{T+K} |\varphi(t)|^2 dt+\int_G^T \left(\int_t^T  |g_0 (t,s)| ds\right)^2 dt +\int_G^{T+K} \int_T^{T+K} |\eta (t,s)|^2 dsdt \right].
\end{aligned}
\end{equation*}
Therefore, the well-posedness of adapted M-solutions to ABSVIE (\ref{1})
follows immediately from the above similar procedures,
and the estimates (\ref{7})-(\ref{8}) also hold.
\end{proof}

\begin{remark}
Inspired by the problem of optimal consumption rates in a
market with delay discussed in \cite{Agram2013},
we introduce the average time-advanced function, and then study the well-posedness of adapted M-solutions to ABSVIE (\ref{1}).
In the proof of Theorem \ref{th1}, by using the norm (\ref{7}) defined on $\mathcal{M}^2 [0,T+K]$
and the existence and uniqueness of (\ref{13}),
there is no need to assume $\beta>0$ as in \cite{Wen2020} in proving the well-posedness to ABSVIE (\ref{1})
so that the conditions of Theorem 3.2 in \cite{Wen2020} have been weakened.
\end{remark}

\section{Comparison theorem for ABSVIEs}\label{sec4}

The aim of this section is to establish the comparison theorem for one-dimensional ABSVIEs,
which will be achieved by Theorem \ref{th1}.

For $i=1,2$, we study the following ABSVIE:
\begin{align}\label{18}
\left\{\begin{array}{lcl} Y_i(t) = \varphi_i (t)+\int_{t}^{T} g_i\left( t, s, Y_i(s), Z_i(t, s), Y_i(s+\delta (s)),
 \int_{s}^{s+\delta (s)} e^{\lambda (s-\theta)} Y_i (\theta)d\theta
 \right)ds\\
\ \ \ \ \ \ \  \  \ -\int_{t}^{T} Z_i (t, s) dW(s),\ \ \
  t\in [0, T],\\
Y_i (t) = \varphi_i (t), \ \ t\in [T, T+K].
\end{array}
\right.
\end{align}
It should be mentioned that the generator $g_i$ is independent of $Z_i (s,t)$, $Z_i (s+\zeta(s), t)$ and $\int_{s}^{s+\zeta (s)} e^{\lambda (s-\theta)} Z_i (\theta, t)d\theta$, so we no longer need to define an adapted M-solution for (\ref{18}).
It is easy to observe that under Assumptions $(A1)$-$(A2)$, for any given $\varphi_i (\cdot) \in \hat{L}_{\mathcal{F}_{T+K}}^{2} (0,T+K; \mathbb{R})$, ABSVIE (\ref{18}) admits
a unique adapted solution $(Y_i (\cdot), Z_i (\cdot,\cdot))\in L_{\mathbb{F}}^{2} (0,T+K; \mathbb{R})\times L_{\mathbb{F}}^{2} (\Delta; \mathbb{R})$.

\begin{theorem}\label{th2}
For $i=1,2$, let $g_i$ satisfy Assumptions $(A1)$-$(A2)$, and let $(Y_i (\cdot), Z_i (\cdot,\cdot))\in L_{\mathbb{F}}^{2} (0,T+K; \mathbb{R})\times L_{\mathbb{F}}^{2} (\Delta; \mathbb{R})$ be the adapted solution
of (\ref{18}). Assume that $\overline{g}=\overline{g}(t,s,y,z,$ $\alpha,\mu)$ satisfies Assumptions $(A1)$-$(A2)$ such that $y\mapsto \overline{g}(t,s,y,z,\alpha,\mu)$ is nondecreasing. Moreover,
for any $(t,s,y,z,\mu)\in \Delta\times\mathbb{R}\times\mathbb{R}\times L_{\mathcal{F}_{r_4}}^2 (\Omega; \mathbb{R})$,
$\overline{g}(t,s,y,z,\cdot,\mu)$ is increasing, namely,
take any $\alpha_i (\cdot)\in L_{\mathbb{F}}^2 (s, T+K; \mathbb{R})$,
if $\alpha_1 (r_1)\leq \alpha_2 (r_1), a.s., a.e.~r_1\in [s, T+K]$, then
\begin{equation*}
\begin{aligned}
&~~~~~~~~~~\overline{g}(t,s,y,z,\alpha_1 (r_1),\mu)\leq \overline{g}(t,s,y,z,\alpha_2 (r_1),\mu),\\
&\forall (t,y,z,\mu)\in [0,s]\times\mathbb{R}\times\mathbb{R}\times L_{\mathcal{F}_{r_4}}^2 (\Omega; \mathbb{R}),~a.s., a.e.~s\in[0, T].
\end{aligned}
\end{equation*}
For any $(t,s,y,z,\alpha)\in \Delta\times\mathbb{R}\times\mathbb{R}\times L_{\mathcal{F}_{r_1}}^2 (\Omega; \mathbb{R})$,
$\overline{g}(t,s,y,z,\alpha,\cdot)$ is increasing, namely,
take any $\mu_i (\cdot)\in L_{\mathbb{F}}^2 (s, T+K; \mathbb{R})$,
if $\mu_1 (r_4)\leq \mu_2 (r_4), a.s., a.e.~r_4\in [s, T+K]$, then
\begin{equation*}
\begin{aligned}
&~~~~~~~~~~\overline{g}(t,s,y,z,\alpha,\mu_1 (r_4))\leq \overline{g}(t,s,y,z,\alpha,\mu_2 (r_4)),\\
&\forall (t,y,z,\alpha)\in [0,s]\times\mathbb{R}\times\mathbb{R}\times L_{\mathcal{F}_{r_1}}^2 (\Omega; \mathbb{R}),~a.s., a.e.~s\in[0, T].
\end{aligned}
\end{equation*}
Furthermore,
\begin{equation*}
\begin{aligned}
&~~~~~~~~~~~~~g_1 (t,s,y,z,\alpha,\mu)\leq \overline{g}(t,s,y,z,\alpha,\mu)\leq g_2 (t,s,y,z,\alpha,\mu),\\
&\forall (t,y,z,\alpha,\mu)\in [0,s]\times\mathbb{R}\times\mathbb{R}\times L_{\mathcal{F}_{r_1}}^2 (\Omega; \mathbb{R})\times L_{\mathcal{F}_{r_4}}^2 (\Omega; \mathbb{R}),~a.s., a.e.~s\in[0, T].
\end{aligned}
\end{equation*}
Hence, for each $\varphi_i (\cdot)\in \hat{L}_{\mathcal{F}_{T+K}}^{2} (0,T+K; \mathbb{R})$, if $\varphi_1 (t)\leq\varphi_2 (t), a.s., a.e.~t\in[0, T+K]$, then
\begin{equation*}
\begin{aligned}
Y_1 (t)\leq Y_2 (t),~a.s., a.e.~t\in[0, T+K].
\end{aligned}
\end{equation*}
\end{theorem}
\begin{proof}
Suppose that $\overline{\varphi} (\cdot)\in \hat{L}_{\mathcal{F}_{T+K}}^{2} (0,T+K; \mathbb{R})$ satisfies
\begin{equation*}
\begin{aligned}
\varphi_1 (t) \leq \overline{\varphi} (t) \leq \varphi_2 (t),~a.s., a.e.~t\in [0,T+K].
\end{aligned}
\end{equation*}
Thanks to Theorem \ref{th1}, let $(\overline{Y}(\cdot), \overline{Z}(\cdot, \cdot))\in L_{\mathbb{F}}^{2} (0,T+K; \mathbb{R})\times L_{\mathbb{F}}^{2} (\Delta; \mathbb{R})$ be the unique adapted solution to the following equation:
\begin{align*}
\left\{\begin{array}{lcl} \overline{Y}(t) =\overline{ \varphi} (t)+\int_{t}^{T} \overline{g} \left( t, s, \overline{Y}(s), \overline{Z}(t, s), \overline{Y}(s+\delta (s)),
 \int_{s}^{s+\delta (s)} e^{\lambda (s-\theta)} \overline{Y}(\theta)d\theta
 \right)ds\\
\ \ \ \ \ \ \  \  \ -\int_{t}^{T} \overline{Z} (t, s) dW(s),\ \ \
  t\in [0, T],\\
\overline{Y} (t) = \overline{\varphi} (t), \ \ t\in [T, T+K].
\end{array}
\right.
\end{align*}
Set $\widetilde{Y}_0 (\cdot)=Y_2 (\cdot)$. By Theorem $3.3$ in \cite{Wen2020},
there exists a unique adapted solution
$(\widetilde{Y}_1 (\cdot), \widetilde{Z}_1 (\cdot,\cdot))$ $\in L_{\mathbb{F}}^{2} (0,T+K; \mathbb{R})\times L_{\mathbb{F}}^{2} (\Delta; \mathbb{R})$ to the following equation:
\begin{align*}
\left\{\begin{array}{lcl} \widetilde{Y}_1 (t) =\overline{ \varphi} (t)+\int_{t}^{T} \overline{g} \left( t, s, \widetilde{Y}_1 (s), \widetilde{Z}_1 (t, s), \widetilde{Y}_1(s+\delta (s)),
 \int_{s}^{s+\delta (s)} e^{\lambda (s-\theta)} \widetilde{Y}_0 (\theta)d\theta
 \right)ds\\
\ \ \ \ \ \ \  \  \ -\int_{t}^{T} \widetilde{Z}_1 (t, s) dW(s),\ \ \
  t\in [0, T],\\
\widetilde{Y}_1 (t) = \overline{\varphi} (t), \ \ t\in [T, T+K].
\end{array}
\right.
\end{align*}
It should be pointed out that
\begin{equation*}
\begin{aligned}
&\overline{g}\left(t,s,y,z,\alpha, \int_{s}^{s+\delta (s)} e^{\lambda (s-\theta)} \widetilde{Y}_0 (\theta) d\theta\right)
\leq g_2 \left(t,s,y,z,\alpha, \int_{s}^{s+\delta (s)} e^{\lambda (s-\theta)} \widetilde{Y}_0 (\theta) d\theta\right),
\end{aligned}
\end{equation*}
for $(t,y,z,\alpha)\in [0,s] \times\mathbb{R}\times\mathbb{R}\times L_{\mathcal{F}_{r_1}}^2 (\Omega; \mathbb{R}), a.s., a.e.~s\in[0,T]$, and $\overline{\varphi} (t)\leq \varphi_2 (t)$ for $a.s., a.e.~t\in[0,T+K]$.
Then by Theorem $4.1$ in \cite{Wen2020}, we have
\begin{equation}\label{48}
\begin{aligned}
\widetilde{Y}_1 (t)\leq \widetilde{Y}_0 (t)=Y_2 (t),~a.s., a.e.~t\in[0,T+K].
\end{aligned}
\end{equation}
Again, let $(\widetilde{Y}_2 (\cdot), \widetilde{Z}_2 (\cdot,\cdot))\in L_{\mathbb{F}}^{2} (0,T+K; \mathbb{R})\times L_{\mathbb{F}}^{2} (\Delta; \mathbb{R})$ satisfy the following equation:
\begin{align*}
\left\{\begin{array}{lcl} \widetilde{Y}_2 (t) =\overline{ \varphi} (t)+\int_{t}^{T} \overline{g} \left( t, s, \widetilde{Y}_2 (s), \widetilde{Z}_2 (t, s), \widetilde{Y}_2 (s+\delta (s)),
 \int_{s}^{s+\delta (s)} e^{\lambda (s-\theta)} \widetilde{Y}_1 (\theta)d\theta
 \right)ds\\
\ \ \ \ \ \ \  \  \ -\int_{t}^{T} \widetilde{Z}_2 (t, s) dW(s),\ \ \
  t\in [0, T],\\
\widetilde{Y}_2 (t) = \overline{\varphi} (t), \ \ t\in [T, T+K].
\end{array}
\right.
\end{align*}
Since $\mu\mapsto \overline{g}(t,s,y,z,\alpha,\mu)$ is nondecreasing, it follows from (\ref{48}) that
\begin{align*}
&\overline{g}\left(t,s,y,z,\alpha, \int_{s}^{s+\delta (s)} e^{\lambda (s-\theta)} \widetilde{Y}_1 (\theta) d\theta\right)
\leq  \overline{g}\left(t,s,y,z,\alpha, \int_{s}^{s+\delta (s)} e^{\lambda (s-\theta)} \widetilde{Y}_0 (\theta) d\theta\right), \\
&~~~~~~~~~~\forall (t,y,z,\alpha)\in [0, s] \times\mathbb{R}\times\mathbb{R}\times L_{\mathcal{F}_{r_1}}^2 (\Omega; \mathbb{R}),~ a.s., a.e.~s\in [0, T].
\end{align*}
By the similar argument used in the above proof, we arrive at
\begin{align*}
\widetilde{Y}_2 (t)\leq\widetilde{Y}_1 (t),~a.s., a.e.~t\in[0,T+K].
\end{align*}
Clearly, there exists a sequence $\{(\widetilde{Y}_d (\cdot), \widetilde{Z}_d (\cdot,\cdot))\}_{d=1}^{\infty}$
in the Banach space $L_{\mathbb{F}}^{2} (0,T+K; \mathbb{R})\times L_{\mathbb{F}}^{2} (\Delta; \mathbb{R})$ such that
\begin{align}\label{19}
\left\{\begin{array}{lcl} \widetilde{Y}_d (t) =\overline{ \varphi} (t)+\int_{t}^{T} \overline{g} \left( t, s, \widetilde{Y}_d (s), \widetilde{Z}_d (t, s), \widetilde{Y}_d (s+\delta (s)),
 \int_{s}^{s+\delta (s)} e^{\lambda (s-\theta)} \widetilde{Y}_{d-1} (\theta)d\theta
 \right)ds\\
\ \ \ \ \ \ \  \  \ -\int_{t}^{T} \widetilde{Z}_d (t, s) dW(s),\ \ \
  t\in [0, T],\\
\widetilde{Y}_d (t) = \overline{\varphi} (t), \ \ t\in [T, T+K].
\end{array}
\right.
\end{align}
We claim that
\begin{equation*}
\begin{aligned}
Y_2 (t)=\widetilde{Y}_0 (t) \geq\widetilde{Y}_1 (t)\geq\widetilde{Y}_2 (t)\geq\cdots\geq\widetilde{Y}_d (t)\geq\cdots,~a.s., a.e.~t\in[0,T+K].
\end{aligned}
\end{equation*}

Now let us prove that $\{(\widetilde{Y}_d (\cdot), \widetilde{Z}_d (\cdot,\cdot))\}_{d=1}^{\infty}$ is a Cauchy sequence in $L_{\mathbb{F}}^{2} (0,T+K; \mathbb{R})\times L_{\mathbb{F}}^{2} (\Delta; \mathbb{R})$.
To this end, we need to introduce the following new equivalent norm in $L_{\mathbb{F}}^{2} (0,T+K; \mathbb{R})\times L_{\mathbb{F}}^{2} (\Delta; \mathbb{R})$:
\begin{equation*}
\begin{aligned}
\|(y(\cdot),z(\cdot,\cdot))\|_{L_{\mathbb{F}}^{2} (0,T+K; \mathbb{R})\times L_{\mathbb{F}}^{2} (\Delta; \mathbb{R})}^2:=
\mathbb{E}\left[\int_0^{T+K} e^{\beta t} |y(t)|^2 dt+\int_0^{T} e^{\beta t} \int_t^{T}  |z(t,s)|^2 dsdt\right],
\end{aligned}
\end{equation*}
where $(y(\cdot),z(\cdot,\cdot))\in L_{\mathbb{F}}^{2} (0,T+K; \mathbb{R})\times L_{\mathbb{F}}^{2} (\Delta; \mathbb{R})$, and $\beta$ is a undetermined positive constant.
One can easily derive from the stability estimate $(3.45)$ in \cite{Yong2008} that
\begin{equation*}
\begin{aligned}
&~~~\mathbb{E}\left[|\widetilde{Y}_d (t)-\widetilde{Y}_l (t)|^2 +\int_t^T |\widetilde{Z}_d (t,s)-\widetilde{Z}_l (t,s)|^2 ds\right]\\
&\leq \mathbb{E}\left[\left(\int_t^T \left|\overline{g} \left( t, s, \widetilde{Y}_d (s), \widetilde{Z}_d (t, s), \widetilde{Y}_d (s+\delta (s)),
 \int_{s}^{s+\delta (s)} e^{\lambda (s-\theta)} \widetilde{Y}_{d-1} (\theta)d\theta
 \right)\right.\right.\right.\\
&~~~\left.\left.\left.-\overline{g} \left( t, s, \widetilde{Y}_d (s), \widetilde{Z}_d (t, s), \widetilde{Y}_l (s+\delta (s)),
 \int_{s}^{s+\delta (s)} e^{\lambda (s-\theta)} \widetilde{Y}_{l-1} (\theta)d\theta
 \right)   \right| ds \right)^2 \right].
\end{aligned}
\end{equation*}
It is straightforward to check that
\begin{equation*}
\begin{aligned}
&~~~~\mathbb{E}\left[\int_0^T e^{\beta t} |\widetilde{Y}_d (t)-\widetilde{Y}_l (t)|^2 dt
+\int_0^T e^{\beta t}\left( \int_t^T  |\widetilde{Z}_d (t,s)-\widetilde{Z}_l (t,s)|^2 ds\right) dt\right]\\
&\leq \mathbb{E}\left[ \int_0^T e^{\beta t} \left(\int_t^T \left|\overline{g} \left( t, s, \widetilde{Y}_d (s), \widetilde{Z}_d (t, s), \widetilde{Y}_d (s+\delta (s)),
 \int_{s}^{s+\delta (s)} e^{\lambda (s-\theta)} \widetilde{Y}_{d-1} (\theta)d\theta
 \right)\right.\right.\right.\\
&~~~\left.\left.\left.-\overline{g} \left( t, s, \widetilde{Y}_d (s), \widetilde{Z}_d (t, s), \widetilde{Y}_l (s+\delta (s)),
 \int_{s}^{s+\delta (s)} e^{\lambda (s-\theta)} \widetilde{Y}_{l-1} (\theta)d\theta
 \right)   \right| ds \right)^2 dt \right]\\
&\leq \frac{Q}{\beta}\mathbb{E}\left[\int_0^{T+K} e^{\beta t} |\widetilde{Y}_d (t)-\widetilde{Y}_l (t)|^2 dt
+\int_0^T e^{\beta t}\left( \int_t^T  |\widetilde{Z}_d (t,s)-\widetilde{Z}_l (t,s)|^2 ds\right) dt\right. \\
&~~~\left.+\int_0^{T+K} e^{\beta t} |\widetilde{Y}_{d-1} (t)-\widetilde{Y}_{l-1} (t)|^2 dt  \right],
\end{aligned}
\end{equation*}
where $Q=CTM+\frac{CT^2}{2\lambda}\left(1-e^{-2\lambda (T+K)}\right)$. Since $\widetilde{Y}_d (t)-\widetilde{Y}_l (t)=0$ for any $t\in [T, T+K]$, one can obtain that
\begin{equation*}
\begin{aligned}
&~~~~\left(1-\frac{Q}{\beta}\right)\mathbb{E}\left[\int_0^{T+K} e^{\beta t} |\widetilde{Y}_d (t)-\widetilde{Y}_l (t)|^2 dt+\int_0^T e^{\beta t} \left(\int_t^T  |\widetilde{Z}_d (t,s)-\widetilde{Z}_l (t,s)|^2 ds\right)dt\right]\\
&\leq\frac{Q}{\beta}\mathbb{E}\left[\int_0^{T+K} e^{\beta t} |\widetilde{Y}_{d-1} (t)-\widetilde{Y}_{l-1} (t)|^2 dt\right].
\end{aligned}
\end{equation*}
If $\beta=3Q>0$, then the following estimate holds:
\begin{equation*}
\begin{aligned}
&~~~~\mathbb{E}\left[\int_0^{T+K} e^{\beta t} |\widetilde{Y}_d (t)-\widetilde{Y}_l (t)|^2 dt+\int_0^T e^{\beta t} \left(\int_t^T  |\widetilde{Z}_d (t,s)-\widetilde{Z}_l (t,s)|^2 ds\right)dt\right]\\
&\leq\frac{1}{2}\mathbb{E}\left[\int_0^{T+K} e^{\beta t} |\widetilde{Y}_{d-1} (t)-\widetilde{Y}_{l-1} (t)|^2 dt\right]\\
&\leq\frac{1}{2}\mathbb{E}\left[\int_0^{T+K} e^{\beta t} |\widetilde{Y}_{d-1} (t)-\widetilde{Y}_{l-1} (t)|^2 dt
+\int_0^T e^{\beta t} \left(\int_t^T  |\widetilde{Z}_{d-1} (t,s)-\widetilde{Z}_{l-1} (t,s)|^2 ds\right)dt \right].
\end{aligned}
\end{equation*}
This means that $\{(\widetilde{Y}_d (\cdot), \widetilde{Z}_d (\cdot,\cdot))\}_{d=1}^{\infty}$ is a Cauchy sequence in the Banach space $L_{\mathbb{F}}^{2} (0,T+K;\mathbb{R})\times L_{\mathbb{F}}^{2} (\Delta; \mathbb{R})$ and exists a unique limit $(\widetilde{Y} (\cdot), \widetilde{Z} (\cdot,\cdot))\in L_{\mathbb{F}}^{2} (0,T+K;\mathbb{R})\times L_{\mathbb{F}}^{2} (\Delta; \mathbb{R})$ such that
\begin{align*}
\lim_{d \rightarrow\infty} \mathbb{E}\left[\int_0^{T+K} e^{\beta t} |\widetilde{Y}_d (t)-\widetilde{Y} (t)|^2 dt+\int_0^T e^{\beta t} \left(\int_t^T  |\widetilde{Z}_d (t,s)-\widetilde{Z} (t,s)|^2 ds\right)dt\right]=0.
\end{align*}
Therefore, by letting $d\rightarrow\infty$ in (\ref{19}) yields
\begin{align*}
\left\{\begin{array}{lcl} \widetilde{Y} (t) =\overline{ \varphi} (t)+\int_{t}^{T} \overline{g} \left( t, s, \widetilde{Y} (s), \widetilde{Z} (t, s), \widetilde{Y} (s+\delta (s)),
 \int_{s}^{s+\delta (s)} e^{\lambda (s-\theta)} \widetilde{Y} (\theta)d\theta
 \right)ds\\
\ \ \ \ \ \ \  \  \ -\int_{t}^{T} \widetilde{Z} (t, s) dW(s),\ \ \
  t\in [0, T],\\
\widetilde{Y} (t) = \overline{\varphi} (t), \ \ t\in [T, T+K].
\end{array}
\right.
\end{align*}
By the uniqueness of Theorem \ref{th1}, we can show that
\begin{equation*}
\begin{aligned}
Y_2 (t)=\widetilde{Y}_0 (t)\geq\widetilde{Y}(t)= \overline{Y}(t),~a.s., a.e.~t\in[0,T+K].
\end{aligned}
\end{equation*}
Similar to the above discussion, we have the following:
\begin{equation*}
\begin{aligned}
\overline{Y} (t)\geq Y_1 (t),~a.s., a.e.~t\in[0,T+K].
\end{aligned}
\end{equation*}
Thus, it holds immediately
\begin{equation*}
\begin{aligned}
Y_2 (t)\geq Y_1 (t),~a.s., a.e.~t\in[0, T+K].
\end{aligned}
\end{equation*}
\end{proof}

\section{Regularity of adapted M-solutions to ABSVIEs}\label{sec5}

In this section, we would like to prove some regularity results
of the adapted M-solution to ABSVIE (\ref{1}) by virtue of Malliavin calculus.
For simplicity, let us first introduce some spaces with respect to Malliavin calculus.
The related theory on this subject can be found in \cite{Karoui1997,Nualart1995}.

Denote by $\Sigma$ the space of all $\mathcal{F}_{T+K}$-measurable random variables $F$ of the form
$$F=f(W(h_1),\ldots,W(h_{k})),$$
where $f\in C_b^\infty (\mathbb{R}^{k}, \mathbb{R})$, $h_1,\ldots,h_{k}\in L^2(0,T+K;\mathbb{R}^m)$, and
\begin{align*}
W(h_i):=\int_0^{T+K} h_i(t) dW(t).
\end{align*}

For each $F\in \Sigma$, we define
\begin{align*}
D_t F:=\sum\limits_{i=1}^{k} f_{x_i} (W(h_1),\ldots,W(h_{k}))h_i (t),~~0\leq t \leq T+K,
\end{align*}
where $f_{x_i}$ is the derivative of $f$ with respect to its $i$-th variable, and call $D_t F~(0\leq t\leq T+K)$ the Malliavin derivative of $F$ with respect to $W(\cdot)$. Denote the norm on $\Sigma$ by
\begin{align*}
\|F\|_{\mathbb{D}_{1,2}}^2:=\mathbb{E}\left[ |F|^2 +\int_0^{T+K} |D_t F|^2 dt\right]
:=\mathbb{E}\left[ |F|^2 +\int_0^{T+K} \sum\limits_{i=1}^{m} |D_t^i F|^2 dt\right],
\end{align*}
where $D_t^i F$ is the $i$-th component of $D_t F$. Let $\mathbb{D}_{1,2}$ be the closure of $\Sigma$ under the norm $\|\cdot\|_{\mathbb{D}_{1,2}}$, and then $\mathbb{D}_{1,2}$ is a Sobolev space.

Let $\mathbb{Y}^2 [0, T+K]$ be the space of all processes $y(\cdot)\in \hat{L}_{\mathcal{F}_{T+K}}^{\infty} (0,T+K; \mathbb{R}^n)$ such that
\begin{align*}
\|y(\cdot)\|_{\mathbb{Y}^2 [0, T+K]}^2
:=\sup\limits_{t\in [0, T+K]} \mathbb{E}\left[ |y(t)|^2 +\int_t^{T+K} \sum\limits_{i=1}^{m} |D_t^i y(s)|^2 ds\right]<\infty.
\end{align*}
Let $\mathbb{Z}^2 [0, T+K]$ be the space of all processes $z(\cdot,\cdot): \Omega\times [0, T+K]^2 \mapsto \mathbb{R}^{n\times m}$
such that for any $t\in [0, T+K]$, $z(t,\cdot)$ is $\mathbb{F}$-adapted, and
\begin{align*}
\|z(\cdot,\cdot)\|_{\mathbb{Z}^2 [0, T+K]}^2
&:=\sup\limits_{t\in [0, T+K]} \mathbb{E}\left[ \int_0^{T+K} |z(t,s)|^2 ds
+\int_t^{T+K} |z(s,t)|^2 ds\right.\\
&\left.~~~~+\int_t^{T+K}\int_t^{T+K} \sum\limits_{i=1}^{m} \left|D_t^i z(u,s)\right|^2 dsdu\right]<\infty.
\end{align*}
For all $t\in [0,T+K]$ and for each $i=1,2,\cdots,m$, we define $\mathbb{M}^2 [0,T+K]$ to be the space of all pairs $(y(\cdot),z(\cdot,\cdot))\in
\mathbb{Y}^2 [0, T+K]\times \mathbb{Z}^2 [0, T+K]$ such that both $(y(\cdot),z(\cdot,\cdot))$ and
$(D_t^i y(\cdot),D_t^i z(\cdot,\cdot))$ satisfy (\ref{2}).
For any $0\leq A<B \leq T+K$, the spaces $\mathbb{Y}^2 [A, B]$, $\mathbb{Z}^2 [A, B]$ and $\mathbb{M}^2 [A,B]$ can be defined similarly.

To facilitate the discussion of the regularity
to adapted M-solutions, we rewrite ABSVIE (\ref{1}) as follows:
\begin{align}\label{20}
\left\{\begin{array}{lcl} Y(t) = \varphi (t)+\int_{t}^{T} g\left( \Lambda (t,s) \right)ds
 -\int_{t}^{T} Z(t, s) dW(s),\ \ \
  t\in [0, T],\\
Y(t) = \varphi (t), \ \ t\in [T, T+K],\\
Z(t, s) = \eta(t, s), \ \ (t, s)\in [0, T+K]^2 \backslash [0, T]^2,
\end{array}
\right.
\end{align}
where
\begin{align*}
\Lambda (t,s)&:= \left( t, s, Y(s), Z(t, s), Z(s, t), Y(s+\delta (s)),
Z(t, s+\zeta(s)), Z(s+\zeta(s), t),\right.\\
&~~~ \left. \int_{s}^{s+\delta (s)} e^{\lambda (s-\theta)} Y(\theta)d\theta,
\int_{s}^{s+\zeta (s)} e^{\lambda (s-\theta)} Z(t, \theta)d\theta,  \int_{s}^{s+\zeta (s)} e^{\lambda (s-\theta)} Z(\theta, t)d\theta  \right),
\end{align*}
and suppose that the generator $g$ satisfies the following assumptions:
\begin{itemize}
\item[(A3)]
$g:\Omega\times\Delta\times (\mathbb{R}^n \times\mathbb{R}^{n\times m}\times\mathbb{R}^{n\times m})^3 \mapsto \mathbb{R}^n$
is $\mathcal{F}_T\otimes\mathcal{B}(\Delta\times (\mathbb{R}^n \times\mathbb{R}^{n\times m}\times\mathbb{R}^{n\times m})^3)$-measurable
such that $s\mapsto g(t,s,y,z,\xi,\alpha,\beta,\gamma,\mu,\nu, \psi)$ is $\mathbb{F}$-progressively measurable for all
$(t,y,z,\xi,\alpha,\beta,\gamma,\mu,\nu, \psi)\in [0,T]\times(\mathbb{R}^n \times\mathbb{R}^{n\times m}\times\mathbb{R}^{n\times m})^3$.	
For any $(t,s)\in \Delta$, $g(\cdot,t,s,\cdot,\cdot,\cdot,\cdot,\cdot,\cdot,\cdot,\cdot,\cdot,)$
is a map from $\Omega\times \mathbb{R}^n \times\mathbb{R}^{n\times m}\times\mathbb{R}^{n\times m}
\times L_{\mathcal{F}_{r_1}}^2 (\Omega; \mathbb{R}^n)\times L_{\mathcal{F}_{r_2}}^2 (\Omega; \mathbb{R}^{n\times m})
\times L_{\mathcal{F}_{r_3}}^2 (\Omega; \mathbb{R}^{n\times m})\times L_{\mathcal{F}_{r_4}}^2 (\Omega; \mathbb{R}^n)
\times L_{\mathcal{F}_{r_5}}^2 (\Omega; \mathbb{R}^{n\times m})\times L_{\mathcal{F}_{r_6}}^2 (\Omega; \mathbb{R}^{n\times m})$ to
$L_{\mathcal{F}_{s}}^2 (\Omega;\mathbb{R}^n)$ defined by
\begin{equation*}
\begin{aligned}
&~~~~g(\cdot,t,s,\cdot,\cdot,\cdot,\cdot,\cdot,\cdot,\cdot,\cdot,\cdot)(\omega,y,z,\xi,
\alpha,\beta,\gamma,\mu,\nu, \psi)\\
&:=g(\omega,t,s,y,z,\xi,\alpha (\omega),\beta(\omega),\gamma(\omega),\mu(\omega),\nu(\omega), \psi(\omega)),\\
&\forall (\omega,y,z,\xi,
\alpha,\beta,\gamma,\mu,\nu, \psi)\in \Omega\times \mathbb{R}^n \times\mathbb{R}^{n\times m}\times\mathbb{R}^{n\times m}
\times L_{\mathcal{F}_{r_1}}^2 (\Omega; \mathbb{R}^n)\times L_{\mathcal{F}_{r_2}}^2 (\Omega; \mathbb{R}^{n\times m})
\\
&\times L_{\mathcal{F}_{r_3}}^2 (\Omega; \mathbb{R}^{n\times m})\times L_{\mathcal{F}_{r_4}}^2 (\Omega; \mathbb{R}^n)
\times L_{\mathcal{F}_{r_5}}^2 (\Omega; \mathbb{R}^{n\times m})\times L_{\mathcal{F}_{r_6}}^2 (\Omega; \mathbb{R}^{n\times m}),
\end{aligned}
\end{equation*}
where $r_i \in [s, T+K]$, $i=1,\cdots,6$,
and (\ref{46}) holds.
Moreover, for almost all $(\omega,t,s)\in \Omega\times\Delta$,
the map $(y,z,\xi,\alpha,\beta,\gamma,\mu,\nu, \psi)\mapsto g(\omega,t,s,y,z,\xi,\alpha,\beta,\gamma,\mu,$ $\nu, \psi)$ is continuously differentiable, and
all partial derivatives of $g$ are uniformly bounded.
\item[(A4)]
For almost all $(\omega,t,s)\in \Omega\times\Delta$, the map $(y,z,\xi,\alpha,\beta,\gamma,\mu,\nu, \psi)\mapsto \left[D_r^i g\right](\omega,t,s,y,z,\xi,\alpha,$ $\beta,\gamma,\mu,\nu,\psi)$ is continuous for all $r\in [0,T]$ and $i=1,2,\cdots,m$.
There exists an $\mathcal{F}_{T}\otimes\mathcal{B}(\Delta)$-measurable process $L (t,s): \Omega\times\Delta\mapsto [0,+\infty)$ satisfying
\begin{align*}
\sup_{t\in [0,T]}\mathbb{E}\left[\left(\int_t^T |L (t,s)| ds\right)^2 \right]<\infty,
\end{align*}
such that
\begin{align*}
&\sum\limits_{i=1}^{m} \left|\left[D_r^i g\right] (\omega,t,s,y,z,\xi,\alpha,\beta,\gamma,\mu,\nu, \psi)\right|\leq L (\omega,t,s),\\
&\forall (\omega,r,t,s,y,z,\xi,\alpha,\beta,\gamma,\mu,\nu, \psi)
\in \Omega\times[0,T]\times\Delta\times\mathbb{R}^n \times\mathbb{R}^{n\times m}\times\mathbb{R}^{n\times m}
\times L_{\mathcal{F}_{r_1}}^2 (\Omega; \mathbb{R}^n)\\
&\times L_{\mathcal{F}_{r_2}}^2 (\Omega; \mathbb{R}^{n\times m})
\times L_{\mathcal{F}_{r_3}}^2 (\Omega; \mathbb{R}^{n\times m})\times L_{\mathcal{F}_{r_4}}^2 (\Omega; \mathbb{R}^n)
\times L_{\mathcal{F}_{r_5}}^2 (\Omega; \mathbb{R}^{n\times m})\times L_{\mathcal{F}_{r_6}}^2 (\Omega; \mathbb{R}^{n\times m}).
\end{align*}
\end{itemize}

Next, we state and prove the main conclusion of this section, which concerns
the differentiability of the adapted M-solution to ABSVIE (\ref{20}).
\begin{theorem}\label{th3}
Suppose that Assumptions $(A3)$-$(A4)$ hold. For any given $(\varphi (\cdot),\eta(\cdot, \cdot))\in \mathbb{M}^2 [0,T+K]$, let $(Y(\cdot), Z(\cdot,\cdot))\in \mathcal{M}^2 [0,T+K]$ be the corresponding unique adapted M-solution to ABSVIE (\ref{20}).
Then for any $0\leq r \leq T+K$ and $1\leq i\leq m$, $(D_r^i Y(\cdot), D_r^i Z(\cdot,\cdot))$ is the unique
adapted M-solution of the following ABSVIE:
\begin{align}\label{21}
\left\{\begin{array}{lcl}
D_r^i Y(t) =D_r^i \varphi (t)+\int_{t}^{T} \left\{\left[D_r^i g\right]\left(\Lambda (t,s)\right)
+g_y \left(\Lambda (t,s)\right)D_r^i Y(s)+g_\alpha \left(\Lambda (t,s)\right)D_r^i Y(s+\delta (s))\right.\\
~~~~~~~~~~~~+g_\mu \left(\Lambda (t,s)\right) \int_{s}^{s+\delta (s)} e^{\lambda (s-\theta)} D_r^i Y(\theta)d\theta
+\sum\limits_{j=1}^{m}\left( g_{z_j}\left(\Lambda (t,s)\right)D_r^i Z_j (t,s)\right.\\
~~~~~~~~~~~~+g_{\xi_j}\left(\Lambda (t,s)\right)D_r^i Z_j (s,t)
+g_{\beta_j}\left(\Lambda (t,s)\right)D_r^i Z_j (t,s+\zeta(s))
\\
~~~~~~~~~~~~+g_{\gamma_j}\left(\Lambda (t,s)\right)D_r^i Z_j (s+\zeta(s),t)
+g_{\nu_j}\left(\Lambda (t,s)\right) \int_{s}^{s+\zeta (s)} e^{\lambda (s-\theta)}D_r^i Z_j (t, \theta)d\theta\\
~~~~~~~~~~~~\left.\left.+g_{\psi_j}\left(\Lambda (t,s)\right) \int_{s}^{s+\zeta (s)} e^{\lambda (s-\theta)}D_r^i Z_j (\theta, t)d\theta  \right)\right\}ds
 -\int_{t}^{T} D_r^i Z(t, s) dW(s),\\
~~~~~~~~~~~~~  0\leq r< t\leq T,\\
D_r^i Y(t) =D_r^i \varphi (t), \ \ T\leq r< t \leq T+K,\\
D_r^i Z(t, s) =D_r^i \eta(t, s), \ \ (t, s)\in [0, T+K]^2 \backslash [0, T]^2,~~ s\in (r, T+K],
\end{array}
\right.
\end{align}
where $\{Z_j (\cdot,\cdot)\}_{1\leq j \leq m}$ represents the $j$-th column of the matrix $Z(\cdot,\cdot)$.\\
Moreover,
\begin{equation}\label{22}
\begin{aligned}
D_r^i Y(t)=Z_i (t,r)+\int_r^t D_r^i Z (t,s) dW(s),~~~0\leq r< t \leq T+K.
\end{aligned}
\end{equation}
It is easy to observe that
\begin{equation}\label{23}
\begin{aligned}
Z_i (t,r)=\mathbb{E}\left[D_r^i Y(t)~|~\mathcal{F}_r \right],~~~0\leq r< t \leq T+K,~a.s.
\end{aligned}
\end{equation}
On the other hand, this implies that
\begin{align}\label{24}
\left\{\begin{array}{lcl}
Z_i (t,r) =D_r^i \varphi (t)+\int_{r}^{T} \left\{\left[D_r^i g\right]\left(\Lambda (t,s)\right)
+g_y \left(\Lambda (t,s)\right)D_r^i Y(s)+g_\alpha \left(\Lambda (t,s)\right)D_r^i Y(s+\delta (s))\right.\\
~~~~~~~~~~~~+g_\mu \left(\Lambda (t,s)\right) \int_{s}^{s+\delta (s)} e^{\lambda (s-\theta)} D_r^i Y(\theta)d\theta
+\sum\limits_{j=1}^{m}\left( g_{z_j}\left(\Lambda (t,s)\right)D_r^i Z_j (t,s)\right.\\
~~~~~~~~~~~~+g_{\xi_j}\left(\Lambda (t,s)\right)D_r^i Z_j (s,t)
+g_{\beta_j}\left(\Lambda (t,s)\right)D_r^i Z_j (t,s+\zeta(s))
\\
~~~~~~~~~~~~+g_{\gamma_j}\left(\Lambda (t,s)\right)D_r^i Z_j (s+\zeta(s),t)
+g_{\nu_j}\left(\Lambda (t,s)\right) \int_{s}^{s+\zeta (s)} e^{\lambda (s-\theta)}D_r^i Z_j (t, \theta)d\theta\\
~~~~~~~~~~~~\left.\left.+g_{\psi_j}\left(\Lambda (t,s)\right) \int_{s}^{s+\zeta (s)} e^{\lambda (s-\theta)}D_r^i Z_j (\theta, t)d\theta  \right)\right\}ds
 -\int_{r}^{T} D_r^i Z(t, s) dW(s),\\
~~~~~~~~~~~~~  0\leq t\leq r\leq T,\\
D_r^i Y(t) =0, \ \ T\leq t \leq r \leq T+K,\\
D_r^i Z(t, s) =0, \ \ (t, s)\in [0, T+K]^2 \backslash [0, T]^2, ~~s\in [0,r].
\end{array}
\right.
\end{align}
\end{theorem}
\begin{proof}
For the sake of convenience, we only need to prove the case of $n=m=1$.

It should be pointed out that Assumptions $(A3)$-$(A4)$ are stronger than Assumptions $(A1)$-$(A2)$,
for any given $(\varphi (\cdot),\eta(\cdot, \cdot))\in \mathbb{M}^2 [0,T+K]$,
ABSVIE (\ref{20}) has a unique adapted M-solution $(Y (\cdot), Z (\cdot,\cdot))\in\mathcal{M}^2 [0,T+K]$
according to Theorem \ref{th1}.
Thus, there exists a unique adapted M-solution $(\tilde{Y}^{r} (\cdot), \tilde{Z}^{r} (\cdot,\cdot))\in\mathcal{M}^2 [0,T+K]$ satisfying the linear ABSVIE:
\begin{align}\label{25}
\left\{\begin{array}{lcl}
\tilde{Y}^r (t) =D_r \varphi (t)+\int_{t}^{T} \left\{\left[D_r g\right]\left(\Lambda (t,s)\right)
+g_y \left(\Lambda (t,s)\right)\tilde{Y}^r (s)+g_\alpha \left(\Lambda (t,s)\right)\tilde{Y}^r (s+\delta (s))\right.\\
~~~~~~~~~~+g_\mu \left(\Lambda (t,s)\right) \int_{s}^{s+\delta (s)} e^{\lambda (s-\theta)} \tilde{Y}^r (\theta)d\theta
+ g_{z}\left(\Lambda (t,s)\right)\tilde{Z}^r (t,s)
+g_{\xi}\left(\Lambda (t,s)\right)\tilde{Z}^r (s,t)\\
~~~~~~~~~~
+g_{\beta}\left(\Lambda (t,s)\right)\tilde{Z}^r (t,s+\zeta(s))
+g_{\gamma}\left(\Lambda (t,s)\right)\tilde{Z}^r (s+\zeta(s),t)
\\
~~~~~~~~~~
\left.+g_{\nu}\left(\Lambda (t,s)\right) \int_{s}^{s+\zeta (s)} e^{\lambda (s-\theta)}\tilde{Z}^r (t, \theta)d\theta
+g_{\psi}\left(\Lambda (t,s)\right) \int_{s}^{s+\zeta (s)} e^{\lambda (s-\theta)}\tilde{Z}^r (\theta, t)d\theta  \right\}ds\\
~~~~~~~~~~ -\int_{t}^{T} \tilde{Z}^r (t, s) dW(s),~  0\leq r< t\leq T,\\
\tilde{Y}^r (t) =D_r \varphi (t), \ \ T\leq r< t \leq T+K,\\
\tilde{Z}^r (t, s) =D_r \eta(t, s), \ \ (t, s)\in [0, T+K]^2 \backslash [0, T]^2, ~~ s\in (r, T+K].
\end{array}
\right.
\end{align}

Next, by Theorem \ref{th1}, we know that if $T-H >0$ is small enough,
then $\Pi$ defined by (\ref{47}) is a contractive map on the Banach space $\mathcal{M}^2 [H,T+K]$.
Letting $(Y_0 (\cdot), Z_0 (\cdot,\cdot))=0$, we can construct a Picard iteration sequence in $\mathcal{M}^2 [H,T+K]$
as follows:
\begin{align*}
(Y_{p+1} (\cdot), Z_{p+1} (\cdot,\cdot))=\Pi (Y_{p} (\cdot), Z_{p} (\cdot,\cdot)),~ p=0,1,2,\cdots.
\end{align*}
Then $\{(Y_{p} (\cdot), Z_{p} (\cdot,\cdot))\}_{p=0}^{\infty}$ converges to the
limit $(Y (\cdot), Z (\cdot,\cdot))$, which is the unique adapted M-solution
to ABSVIE (\ref{20}) such that
\begin{equation}\label{28}
\begin{aligned}
\lim\limits_{p\rightarrow \infty}\|(Y_p (\cdot),Z_p (\cdot,\cdot))-(Y(\cdot),Z(\cdot,\cdot))\|_{\mathcal{M}^2 [H,T+K]}=0.
\end{aligned}
\end{equation}
If $\{(Y_{p} (\cdot), Z_{p} (\cdot,\cdot))\}_{p=0}^{\infty}$ satisfies the following ABSVIE:
\begin{align}\label{57}
\left\{\begin{array}{lcl} Y_{p+1} (t) = \varphi (t)+\int_{t}^{T} g\left( \Lambda_p (t,s) \right)ds-\int_{t}^{T} Z_{p+1}(t, s) dW(s),\ \ \
  t\in [H, T],\\
Y_{p+1}(t) = \varphi (t), \ \ t\in [T, T+K],\\
Z_{p+1}(t, s) = \eta(t, s), \ \ (t, s)\in [H, T+K]^2 \backslash [H, T]^2,
\end{array}
\right.
\end{align}
where
\begin{align*}
\Lambda_p (t,s)&:= \left( t, s, Y_p (s), Z_{p}(t, s), Z_p (s, t), Y_p (s+\delta (s)),
Z_{p}(t, s+\zeta(s)), Z_p (s+\zeta(s), t),\right.\\
&~~~ \left. \int_{s}^{s+\delta (s)} e^{\lambda (s-\theta)} Y_p (\theta)d\theta,
\int_{s}^{s+\zeta (s)} e^{\lambda (s-\theta)} Z_{p}(t, \theta)d\theta,  \int_{s}^{s+\zeta (s)} e^{\lambda (s-\theta)} Z_p (\theta, t)d\theta  \right),
\end{align*}
then by induction, we shall prove that
\begin{align*}
(Y_{p} (\cdot),Z_{p}(\cdot, \cdot))\in
\mathbb{Y}^2 [H, T+K]\times \mathbb{Z}^2 [H, T+K],~~p=0,1,2,\cdots.
\end{align*}
Indeed, suppose that $(Y_{p} (\cdot),Z_{p}(\cdot, \cdot))\in
\mathbb{Y}^2 [H, T+K]\times \mathbb{Z}^2 [H, T+K]$,
and then we shall verify that $(Y_{p+1} (\cdot),Z_{p+1}(\cdot, \cdot))\in
\mathbb{Y}^2 [H, T+K]\times \mathbb{Z}^2 [H, T+K]$.
Since $(\varphi (\cdot),\eta(\cdot, \cdot))\in \mathbb{M}^2 [H,T+K]$ and $(Y_{p} (\cdot),Z_{p}(\cdot, \cdot))\in
\mathbb{Y}^2 [H, T+K]\times \mathbb{Z}^2 [H, T+K]$, it follows that
\begin{align*}
\varphi (t)+\int_{t}^{T} g\left( \Lambda_p (t,s) \right)ds\in \mathbb{D}_{1,2},~~t\in [H, T].
\end{align*}
Thus,
\begin{align*}
Y_{p+1} (t) =\mathbb{E}\left[\varphi (t)+\int_{t}^{T} g\left( \Lambda_p (t,s) \right)ds ~\bigg|~ \mathcal{F}_t \right]\in \mathbb{D}_{1,2},~~t\in [H, T].
\end{align*}
From (\ref{57}) we have the following:
\begin{align*}
\int_{t}^{T} Z_{p+1}(t, s) dW(s) =\varphi (t)+\int_{t}^{T} g\left( \Lambda_p (t,s) \right)ds-Y_{p+1} (t)\in \mathbb{D}_{1,2},~~t\in [H, T].
\end{align*}
It then follows from Lemma 5.1 in \cite{Karoui1997} that $Z_{p+1}(\cdot,\cdot)\in \mathbb{Z}^2 [H, T]$.
Note that $Y_{p+1}(t) = \varphi (t)$ for any $t\in [T, T+K]$, and
$Z_{p+1}(t, s) = \eta(t, s)$ for any $(t, s)\in [H, T+K]^2 \backslash [H, T]^2$, where
$(\varphi (\cdot),\eta(\cdot, \cdot))\in \mathbb{M}^2 [H,T+K]$.
This implies that $Z_{p+1}(\cdot,\cdot)\in \mathbb{Z}^2 [H, T+K]$,
and by Lemma 5.1 in \cite{Karoui1997}, we can deduce that
\begin{align}\label{26}
\left\{\begin{array}{lcl}
D_r Y_{p+1}(t) =D_r \varphi (t)+\int_{t}^{T} \left\{\left[D_r g\right]\left(\Lambda_p (t,s)\right)
+g_y \left(\Lambda_p (t,s)\right)D_r Y_p (s)+g_\alpha \left(\Lambda_p (t,s)\right)D_r Y_p (s+\delta (s))\right.\\
~~~~~~~~~~~~~~~+g_\mu \left(\Lambda_p (t,s)\right) \int_{s}^{s+\delta (s)} e^{\lambda (s-\theta)} D_r Y_p (\theta)d\theta
+ g_{z}\left(\Lambda_p (t,s)\right)D_r Z_{p} (t,s)\\
~~~~~~~~~~~~~~~+g_{\xi}\left(\Lambda_p (t,s)\right)D_r Z_p (s,t)
+g_{\beta}\left(\Lambda_p (t,s)\right)D_r Z_{p} (t,s+\zeta(s))
\\
~~~~~~~~~~~~~~~+g_{\gamma}\left(\Lambda_p (t,s)\right)D_r Z_p (s+\zeta(s),t)
+g_{\nu}\left(\Lambda_p (t,s)\right) \int_{s}^{s+\zeta (s)} e^{\lambda (s-\theta)}D_r Z_{p} (t, \theta)d\theta\\
~~~~~~~~~~~~~~~\left.+g_{\psi}\left(\Lambda_p (t,s)\right) \int_{s}^{s+\zeta (s)} e^{\lambda (s-\theta)}D_r Z_p (\theta, t)d\theta  \right\}ds
 -\int_{t}^{T} D_r Z_{p+1}(t, s) dW(s),\\
~~~~~~~~~~~~~~~~  H\leq r< t\leq T,\\
D_r Y_{p+1}(t) =D_r \varphi (t), \ \ T\leq r< t \leq T+K,\\
D_r Z_{p+1}(t, s) =D_r \eta(t, s), \ \ (t, s)\in [H, T+K]^2 \backslash [H, T]^2, ~~ s\in (r, T+K].
\end{array}
\right.
\end{align}
Under Assumptions $(A3)$-$(A4)$, from Theorem \ref{th1}, it is easy to see that if $T-H >0$ is small enough,
then (\ref{26}) admits a unique adapted M-solution $(D_r Y_{p+1} (\cdot), D_r Z_{p+1} (\cdot,\cdot))\in\mathcal{M}^2 [H,T+K]$.
Hence, $(Y_{p} (\cdot),Z_{p}(\cdot, \cdot))\in
\mathbb{Y}^2 [H, T+K]\times \mathbb{Z}^2 [H, T+K],~p=0,1,2,\cdots$.

Now, let us show that $\{(Y_{p} (\cdot), Z_{p} (\cdot,\cdot))\}_{p=0}^{\infty}$ converges to $(\tilde{Y}^{r} (\cdot), \tilde{Z}^{r} (\cdot,\cdot))$ in $\mathcal{M}^2 [H,T+K]$.
Notice that $D_r Y_{p+1} (t)- \tilde{Y}^{r}(t)=0$ for any $t\in [T, T+K]$, and $D_r Z_{p+1} (t,s)- \tilde{Z}^{r} (t,s)=0$ for any $(t,s)\in [H, T+K]^2 \backslash [H, T]^2$.
From the stability estimate (\ref{8}), by combing (\ref{25}) with (\ref{26}) we have the following:
\begin{equation*}
\begin{aligned}
&~~~\|(D_r Y_{p+1} (\cdot), D_r Z_{p+1} (\cdot,\cdot))-(\tilde{Y}^{r} (\cdot), \tilde{Z}^{r} (\cdot,\cdot))\|_{\mathcal{M}^2 [H,T+K]}^2\\
&:=\mathbb{E}\left[\int_H^{T+K}  |D_r Y_{p+1} (t)-\tilde{Y}^r (t)|^2 dt+\int_H^{T+K} \int_t^{T+K} |D_r Z_{p+1} (t,s)-\tilde{Z}^r (t,s)|^2 dsdt\right]\\
&\leq C \mathbb{E}\left\{
\int_H^T \left[\int_t^T  \left( \left|\left[D_r g\right]\left(\Lambda_p (t,s)\right)-\left[D_r g\right]\left(\Lambda (t,s)\right) \right|
+\left|g_y \left(\Lambda_p (t,s)\right)-g_y \left(\Lambda (t,s)\right)\right| \left|\tilde{Y}^r (s)\right|\right.\right.\right.\\
&~~~+\left|g_z \left(\Lambda_p (t,s)\right)-g_z \left(\Lambda (t,s)\right)\right| \left|\tilde{Z}^r (t,s)\right|
+\left|g_\xi \left(\Lambda_p (t,s)\right)-g_\xi \left(\Lambda (t,s)\right)\right| \left|\tilde{Z}^r (s,t)\right|\\
&~~~+\left|g_\alpha \left(\Lambda_p (t,s)\right)-g_\alpha \left(\Lambda (t,s)\right)\right| \left|\tilde{Y}^r (s+\delta(s))\right|
+\left|g_\beta \left(\Lambda_p (t,s)\right)-g_\beta \left(\Lambda (t,s)\right)\right| \left|\tilde{Z}^r (t,s+\zeta(s))\right|\\
&~~~+\left|g_\gamma \left(\Lambda_p (t,s)\right)-g_\gamma \left(\Lambda (t,s)\right)\right| \left|\tilde{Z}^r (s+\zeta(s),t)\right|
+\left|g_\mu \left(\Lambda_p (t,s)\right)-g_\mu \left(\Lambda (t,s)\right)\right| \\
&~~~\cdot \int_{s}^{s+\delta (s)} e^{\lambda (s-\theta)}\left|\tilde{Y}^r (\theta)\right|d\theta
+\left|g_\nu \left(\Lambda_p (t,s)\right)-g_\nu \left(\Lambda (t,s)\right)\right|
\int_{s}^{s+\zeta (s)} e^{\lambda (s-\theta)}\left|\tilde{Z}^r (t,\theta)\right|d\theta
\end{aligned}
\end{equation*}
\begin{equation}\label{27}
\begin{aligned}
&~~~\left.\left.\left.+\left|g_\psi \left(\Lambda_p (t,s)\right)-g_\psi \left(\Lambda (t,s)\right)\right|
\int_{s}^{s+\zeta (s)} e^{\lambda (s-\theta)}\left|\tilde{Z}^r (\theta,t)\right|d\theta\right)
ds\right]^2dt\right\}\\
&~~~+C(T-H)^2 \|(D_r Y_{p} (\cdot), D_r Z_{p} (\cdot,\cdot))-(\tilde{Y}^{r} (\cdot), \tilde{Z}^{r} (\cdot,\cdot))\|_{\mathcal{M}^2 [H,T+K]}^2 \\
&:=\kappa_p+C(T-H)^2 \|(D_r Y_{p} (\cdot), D_r Z_{p} (\cdot,\cdot))-(\tilde{Y}^{r} (\cdot), \tilde{Z}^{r} (\cdot,\cdot))\|_{\mathcal{M}^2 [H,T+K]}^2.
\end{aligned}
\end{equation}
Let $T-H>0$ be small enough such that
\begin{equation*}
\begin{aligned}
C(T-H)^2<1.
\end{aligned}
\end{equation*}
Thus, under Assumptions $(A3)$-$(A4)$, by the dominated convergence theorem and (\ref{28})
one can show that
\begin{equation*}
\begin{aligned}
\lim\limits_{p\rightarrow\infty}\kappa_p=0.
\end{aligned}
\end{equation*}
By (\ref{27}) it is directly concluded that
\begin{equation*}
\begin{aligned}
\lim\limits_{p\rightarrow\infty}\|(D_r Y_{p} (\cdot), D_r Z_{p} (\cdot,\cdot))-(\tilde{Y}^{r} (\cdot), \tilde{Z}^{r} (\cdot,\cdot))\|_{\mathcal{M}^2 [H,T+K]}=0.
\end{aligned}
\end{equation*}
By the way, $D_r$ is a closed operator, which means that
\begin{equation*}
\begin{aligned}
\tilde{Y}^r (t)=D_r Y (t),~~\tilde{Z}^r (t,s)=D_r Z(t,s),~~~t,s\in [H,T+K], a.s.,
\end{aligned}
\end{equation*}
and (\ref{21})-(\ref{24}) hold when $T-H>0$ is small enough.
By induction, the proof
of (\ref{21})-(\ref{24}) follows immediately from Theorem \ref{th1}.
\end{proof}

%

\section{Nonzero-sum stochastic differential games}\label{sec6}
In this section, as an application of ABSVIEs, we study a nonzero-sum
differential game system of SDVIEs with the aim of
establishing Pontryagin's type maximum principle.

The state of game systems is modelled by the following
SDVIE:
\begin{align}\label{29}
\left\{\begin{array}{lcl} X(t) = \phi (t)+\int_{0}^{t} b\left( t, s, X(s), X(s-\delta),
u_1 (s), u_1 (s-\delta_1), u_2 (s), u_2 (s-\delta_2)\right)ds\\
\ \ \ \ \ \ \  \ \ \ +\int_{0}^{t} \sigma \left( t, s, X(s),
u_1 (s),  u_2 (s)\right) dW(s),\ \
  t\in [0, T],\\
X(t) = \phi (t), \ \ t\in [-\delta, 0],\\
u_1(t) = \vartheta_1 (t), \ \ t\in [-\delta_1, 0],\\
u_2(t) = \vartheta_2 (t), \ \ t\in [-\delta_2, 0],
\end{array}
\right.
\end{align}
where $\phi(\cdot)\in L_{\mathbb{F}}^{2} (-\delta,T; \mathbb{R})$, and
$\vartheta_i (\cdot)\in L_{\mathbb{F}}^{2} (-\delta_i,0; \mathbb{R})$, $i=1,2$.
$b: \Omega\times \Delta^c \times \mathbb{R} \times\mathbb{R}\times\mathbf{U}_1
\times\mathbf{U}_1\times\mathbf{U}_2\times\mathbf{U}_2
\mapsto\mathbb{R}$ is $\mathcal{F}_T\otimes\mathcal{B}(\Delta^c\times \mathbb{R} \times\mathbb{R}\times\mathbf{U}_1
\times\mathbf{U}_1\times\mathbf{U}_2\times\mathbf{U}_2)$-measurable
such that $s\mapsto b(t,s, x, x_\delta, u_1, u_{1\delta_1}, u_2, u_{2\delta_2})
$ is $\mathbb{F}$-progressively measurable
for all $(t,x, x_\delta, u_1, u_{1\delta_1}, u_2, u_{2\delta_2})\in [0,T]\times \mathbb{R}\times\mathbb{R}\times\mathbf{U}_1
\times\mathbf{U}_1\times\mathbf{U}_2\times\mathbf{U}_2$. $\sigma: \Omega\times \Delta^c \times \mathbb{R} \times\mathbf{U}_1
\times\mathbf{U}_2
\mapsto\mathbb{R}$ is $\mathcal{F}_T\otimes\mathcal{B}(\Delta^c\times \mathbb{R} \times\mathbf{U}_1
\times\mathbf{U}_2)$-measurable
such that $s\mapsto
\sigma (t,s,x, u_1, u_2)$ is $\mathbb{F}$-progressively measurable
for all $(t,x,u_1,  u_2)\in [0,T]\times \mathbb{R}\times\mathbf{U}_1
\times\mathbf{U}_2$.	
Here, $\mathbf{U}_i$ is a non-empty convex subset of $\mathbb{R}$
and $u_i(\cdot)$ is the control process
of Player $i$, $i=1,2$.
Moreover, the constants $\delta,\delta_1,\delta_2>0$ are given finite time delays
for processes $X(\cdot),u_1(\cdot),u_2(\cdot)$, respectively,
and there exists a constant $K\geq0$ such that $0<\delta,\delta_1,\delta_2\leq K$.

For $i=1,2$, we define the admissible control set for Player $i$ by
\begin{align*}
\mathcal{A}_i =\left\{u_i: \Omega\times[0,T]\mapsto \mathbf{U}_i ~\Big|~ u_i(\cdot)~\mbox{is}~\mathbb{F}\mbox{-progressively measurable,}~\mathbb{E}\left[\int_0^T |u_i(t)|^2 dt \right]<\infty\right\}.
\end{align*}
Each element of $\mathcal{A}_i$ is said to be an admissible control for Player $i$, and
$\mathcal{A}_1 \times\mathcal{A}_2$ the set of admissible controls for the players.

The performance functional is defined by
\begin{align}\label{30}
J_i(u_1(\cdot),u_2(\cdot))=\mathbb{E}\left[\int_{0}^{T}
h_i(t,X(t),X(t-\delta),
u_1 (t), u_1 (t-\delta_1), u_2 (t), u_2 (t-\delta_2))dt\right],~i=1,2,
\end{align}
where
$
h_i: \Omega\times [0, T] \times \mathbb{R} \times\mathbb{R}\times\mathbf{U}_1
\times\mathbf{U}_1\times\mathbf{U}_2\times\mathbf{U}_2
\mapsto\mathbb{R}
$
is an $\mathcal{F}_T\otimes\mathcal{B}([0,T]\times \mathbb{R} \times\mathbb{R}\times\mathbf{U}_1
\times\mathbf{U}_1\times\mathbf{U}_2\times\mathbf{U}_2)$-measurable map such that
\begin{equation*}
\begin{aligned}
&\mathbb{E}\left[\int_{0}^{T}
\left|h_i (t,X(t),X(t-\delta),
u_1 (t), u_1 (t-\delta_1), u_2 (t), u_2 (t-\delta_2))\right|dt\right]
<\infty,\\
&\forall~ (u_1 (\cdot),u_2 (\cdot))\in \mathcal{A}_1 \times\mathcal{A}_2.
\end{aligned}
\end{equation*}
In addition, let $b$, $\sigma$ and $h_i~(i=1,2)$ satisfy the following assumptions:
\begin{itemize}
\item[(A5)] $b(\cdot, \cdot, 0,0,0,0,0,0), \sigma(\cdot, \cdot, 0,
0,0) \in L_\mathbb{F}^2 (\Delta^c; \mathbb{R})$.
For almost all $(\omega,t,s)\in \Omega\times\Delta^c$, the maps
$(x, x_\delta, u_1,$  $u_{1\delta_1}, u_2, u_{2\delta_2})\mapsto
b(\omega,t,s,x, x_\delta, u_1, u_{1\delta_1}, u_2, u_{2\delta_2})$ and
$(x, u_1, u_2)\mapsto
\sigma (\omega,t,s,x,$ $u_1, u_2)$ are continuously differentiable,
and their partial derivatives $b_x$, $b_{x_\delta}$, $b_{u_1}$, $b_{u_2}$, $b_{u_{1\delta_1}}$,
$b_{u_{2\delta_2}}$, $\sigma_x$, $\sigma_{u_1}$ and $\sigma_{u_2}$ are uniformly bounded.
\item[(A6)]
For almost all $(\omega,t)\in \Omega\times[0,T]$, the map
$(x, x_\delta, u_1, u_{1\delta_1}, u_2, u_{2\delta_2}) \mapsto
h_i(\omega,t, x, x_\delta, u_1, u_{1\delta_1},$ $u_2,u_{2\delta_2})$
is continuously differentiable,
and there exists a constant $C>0$ such that
\begin{equation*}
\begin{aligned}
&|h_{ia}(\omega,t, x, x_\delta, u_1, u_{1\delta_1},u_2,u_{2\delta_2})|\leq
C(1+|x|+|x_\delta|+|u_1|+|u_{1\delta_1}|+|u_2|+|u_{2\delta_2}|),\\
&\forall~(\omega,t, x, x_\delta, u_1, u_{1\delta_1}, u_2, u_{2\delta_2})\in \Omega\times[0,T]\times\mathbb{R} \times\mathbb{R}\times\mathbf{U}_1
\times\mathbf{U}_1\times\mathbf{U}_2\times\mathbf{U}_2,
\end{aligned}
\end{equation*}
where $a=x, x_\delta, u_1, u_{1\delta_1},u_2,u_{2\delta_2}$.
\end{itemize}

Similar to the proof of Theorem 3.10 in \cite{Wang2024b} (or the proof of Theorem 2.3 in \cite{Shi2013}), and from the standard Banach contraction mapping principle, we can obtain that under Assumption $(A5)$, for any given
$\phi(\cdot)\in L_{\mathbb{F}}^{2} (-\delta,T; \mathbb{R})$,
$\vartheta_i(\cdot)\in L_{\mathbb{F}}^{2} (-\delta_i,0; \mathbb{R})~(i=1,2)$
and $(u_1(\cdot),u_2(\cdot))\in\mathcal{A}_1\times\mathcal{A}_2$, SDVIE (\ref{29})
admits a unique adapted solution $X(\cdot)\in L_\mathbb{F}^2 (-\delta,T; \mathbb{R})$.
Now, we can establish the following nonzero-sum differential game
problem driven by SDVIE (\ref{29}).
\begin{problem}\label{prob1}
Find a pair of admissible controls
$(u_1^* (\cdot), u_2^* (\cdot))\in \mathcal{A}_1\times\mathcal{A}_2$ such that
\begin{align}\label{31}
\left\{\begin{array}{lcl} J_1(u_1^* (\cdot), u_2^*(\cdot)) = \inf\limits_{u_1(\cdot)\in\mathcal{A}_1}
J_1(u_1(\cdot), u_2^* (\cdot)),\\
J_2(u_1^*(\cdot), u_2^* (\cdot)) = \inf\limits_{u_2(\cdot)\in\mathcal{A}_2} J_2(u_1^* (\cdot), u_2(\cdot)).
\end{array}
\right.
\end{align}
\end{problem}
\noindent If the pair of admissible controls $(u_1^* (\cdot), u_2^* (\cdot))$ satisfies (\ref{31}), it is then called
a Nash equilibrium point of Problem \ref{prob1},
and the corresponding $X^* (\cdot)$ is said to be an optimal
state process. As usual, $(X^* (\cdot), u^*_1 (\cdot), u^*_2 (\cdot))$ is called an optimal triple of Problem \ref{prob1}.

To prove the maximum principle for Problem \ref{prob1},
we first discuss the duality principle between linear SDVIEs and linear ABSVIEs.
In the following Theorem \ref{th4}, for $i=1,2,3$, it should be mentioned that $A_i (\cdot,\cdot)$ is defined on $[0, T]^2$, then $A_i (t,s)=0$ for any $(t,s)\notin [0, T]^2$.
\begin{theorem}\label{th4}
For $i=1,2,3$, let $A_i (\cdot,\cdot)\in L^\infty(0,T; L_{\mathbb{F}}^{\infty} (0,T; \mathbb{R}))$, $\phi(\cdot)\in L_{\mathbb{F}}^{2} (0,T; \mathbb{R})$ and $\varphi(\cdot)\in L_{\mathcal{F}_T}^{2} (0,T; \mathbb{R})$.
Assume that $X(\cdot)\in L_\mathbb{F}^2 (-\delta,T;\mathbb{R})$ is the adapted solution
to the linear SDVIE:
\begin{align}\label{32}
\left\{\begin{array}{lcl} X(t) = \phi (t)+\int_{0}^{t}\left( A_1(t,s)X(s)+A_2(t,s)X(s-\delta)\right)ds+\int_{0}^{t}  A_3(t,s)X(s)  dW(s),\
  t\in [0, T],\\
X(t) = 0, \ \ t\in [-\delta, 0],
\end{array}
\right.
\end{align}
and $(Y(\cdot), Z(\cdot,\cdot))\in \mathcal{M}^2 [0,T+\delta]$ is the adapted M-solution
to the linear ABSVIE:
\begin{align}\label{33}
\left\{\begin{array}{lcl} Y(t) = \varphi (t)+\int_{t}^{T}\left( A_1(s,t)Y(s)+\mathbb{E}\left[A_2(s+\delta,t+\delta)Y(s+\delta)1_{[0,T-\delta]^2}(s,t)\big|\mathcal{F}_s\right]\right.\\
\ \ \ \ \ \ \  \  \
 \left.+A_3(s,t)Z(s,t)\right)ds-\int_{t}^{T}  Z(t,s) dW(s),\ \
  t\in [0, T],\\
Y(t) = 0, \ \ t\in [T, T+\delta],\\
Z(t,s)=0, \ \ (t, s)\in [0, T+\delta]^2 \backslash [0, T]^2.
\end{array}
\right.
\end{align}
Then
\begin{equation}\label{34}
\begin{aligned}
\mathbb{E}\left[\int_{0}^{T}\varphi(t)X(t)dt \right]=\mathbb{E}\left[\int_{0}^{T}\phi(t)Y(t)dt \right].
\end{aligned}
\end{equation}
\end{theorem}
\begin{proof}
By (\ref{32}), we can show that
\begin{equation}\label{35}
\begin{aligned}
&~~~\mathbb{E}\left[\int_{0}^{T}\phi(t)Y(t)dt \right]\\
&=\mathbb{E}\left[\int_{0}^{T}X(t)Y(t)dt \right]
-\mathbb{E}\left\{\int_{0}^{T}Y(t)\left[\int_{0}^{t}\left(A_1(t,s)X(s)+A_2(t,s)X(s-\delta)
  \right)ds\right]   dt\right\}\\
&~~~-\mathbb{E}\left[\int_{0}^{T} Y(t)\left( \int_0^t A_3(t,s)X(s)dW(s) \right) dt\right]\\
&=\mathbb{E}\left[\int_{0}^{T}X(t)Y(t)dt \right]
-\mathbb{E}\left[\int_{0}^{T}\int_{t}^{T}Y(s)\left(A_1(s,t)X(t)+A_2(s,t)X(t-\delta)
  \right)dsdt\right]\\
&~~-\mathbb{E}\left[\int_{0}^{T} \left(\mathbb{E}[Y(t)]+\int_0^t Z(t,s) dW(s)\right) \left( \int_0^t A_3(t,s)X(s)
 dW(s) \right) dt\right].
\end{aligned}
\end{equation}
With the help of Fubini's theorem, we derive
\begin{equation}\label{36}
\begin{aligned}
&~~~\mathbb{E}\left[\int_{0}^{T}\int_{t}^{T}Y(s)\left(A_1(s,t)X(t)+A_2(s,t)X(t-\delta)
  \right)dsdt\right]\\
&=\mathbb{E}\left[\int_{0}^{T} X(t)\left(\int_{t}^{T} A_1(s,t)Y(s)ds\right)dt\right]
+\mathbb{E}\left[\int_{0}^{T}X(t-\delta)\left(\int_{t}^{T} A_2(s,t) Y(s)ds  \right) dt\right]\\
&=\mathbb{E}\left[\int_{0}^{T} X(t)\left(\int_{t}^{T} A_1(s,t)Y(s)ds\right)dt\right]\\
&~~~+\mathbb{E}\left[\int_{-\delta}^{T-\delta}X(t)\left(\int_{t}^{T-\delta}A_2(s+\delta,t+\delta) Y(s+\delta)1_{[0,T-\delta]^2}(s,t)ds\right)dt \right]\\
&=\mathbb{E}\left[\int_{0}^{T} X(t)\left(\int_{t}^{T} A_1(s,t)Y(s)ds\right)dt\right]\\
&~~~+\mathbb{E}\left[\int_{0}^{T}X(t)\left(\int_{t}^{T}\mathbb{E}\left[A_2(s+\delta,t+\delta) Y(s+\delta)1_{[0,T-\delta]^2}(s,t)\big|\mathcal{F}_s\right]ds\right)dt \right].
\end{aligned}
\end{equation}
From the property of the It\^{o} integral, we have the following:
\begin{equation}\label{37}
\begin{aligned}
&~~~\mathbb{E}\left[\int_{0}^{T} \left(\mathbb{E}[Y(t)]+\int_0^t Z(t,s) dW(s)\right) \left( \int_0^t A_3(t,s)X(s)
 dW(s) \right) dt\right]\\
&=\mathbb{E}\left[\int_{0}^{T} \int_0^t  A_3(t,s)X(s)Z(t,s)dsdt\right]
=\mathbb{E}\left[\int_{0}^{T} \int_t^T  A_3(s,t)X(t)Z(s,t) dsdt\right]\\
&=\mathbb{E}\left[\int_{0}^{T} X(t) \left(\int_t^T A_3(s,t)Z(s,t) ds\right) dt\right]
\end{aligned}
\end{equation}
Substituting (\ref{36})-(\ref{37}) into (\ref{35}), it then follows from (\ref{33}) that
\begin{equation*}
\begin{aligned}
&~~~\mathbb{E}\left[\int_{0}^{T}\phi(t)Y(t)dt \right]\\
&=\mathbb{E}\left\{\int_{0}^{T}X(t)\left[Y(t)-\int_{t}^{T}\left( A_1(s,t)Y(s)+\mathbb{E}\left[A_2(s+\delta,t+\delta)Y(s+\delta)1_{[0,T-\delta]^2}(s,t)\big|\mathcal{F}_s\right]\right.\right.\right.\\
&~~~\left.\left.\left. +A_3(s,t)Z(s,t)\right) ds \right]dt\right\}\\
&=\mathbb{E}\left[\int_{0}^{T}X(t)\left(\varphi (t)- \int_{t}^{T} Z(t,s) dW(s) \right)dt\right]\\
&=\mathbb{E}\left[\int_{0}^{T}\varphi(t)X(t)dt \right],
\end{aligned}
\end{equation*}
which is the required (\ref{34}).
\end{proof}

For each optimal triple $(X^* (\cdot), u_1^* (\cdot), u_2^* (\cdot))$, we then investigate the Pontryagin type maximum principle for Problem \ref{prob1}. To simplify, we denote the partial derivatives of $b(\cdot)$, $\sigma(\cdot)$ and $h_i(\cdot)~(i=1,2)$ by the following notations:
\begin{align*}
&\sigma^*_{\tilde{l}} (t,s)=\sigma_{\tilde{l}} (t, s, X^*(s),
u_1^* (s),  u_2^* (s)),~~
\sigma^*_{\tilde{l}} (s,t)=\sigma_{\tilde{l}} (s, t, X^*(t),
u_1^* (t),  u_2^* (t)),\\
&b^*_l (t,s)=b_l (t, s, X^*(s), X^*(s-\delta),
u_1^* (s), u_1^* (s-\delta_1), u_2^* (s), u_2^* (s-\delta_2)),\\
&b^*_l (s,t)=b_l (s, t, X^*(t), X^*(t-\delta),
u_1^* (t), u_1^* (t-\delta_1), u_2^* (t), u_2^* (t-\delta_2)),\\
&h^*_{il} (t)=h_{il} (t, X^*(t),X^*(t-\delta),
u_1^* (t), u_1^* (t-\delta_1), u_2^* (t), u_2^* (t-\delta_2)),
\end{align*}
where $\tilde{l}=x,u_1,u_2$, $l=x,x_\delta,u_1,u_{1\delta_1},u_2,u_{2\delta_2}$,
and $t,s\in [0,T]$. Since $h_i(\cdot)$ is defined on $[0, T]$, the values of $h_i (\cdot)$ are equal to zero on $[-\delta,0)$ and $(T,T+\delta]$, $i=1,2$. Similarly, the coefficient $b(\cdot,\cdot)$ is defined on $[0, T]^2$, and then $b(t,s)=0$ for any $(t,s)\notin[0, T]^2$.

\begin{theorem}\label{th5}
If Assumptions $(A5)$-$(A6)$ hold and $(X^* (\cdot), u_1^* (\cdot), u_2^* (\cdot))$ is an optimal triple of Problem \ref{prob1}, then there exists a unique adapted M-solution $(Y_i(\cdot),Y_{0_i} (\cdot); Z_i(\cdot,\cdot),Z_{0_i}(\cdot,\cdot))~(i=1,2)$ satisfying the following ABSVIE:
\begin{align}\label{38}
\left\{\begin{array}{lcl} Y_i (t) = h_{ix}^* (t) +\mathbb{E}\left[h_{ix_{\delta}}^* (t+\delta)1_{[0,T-\delta]}(t)\big| \mathcal{F}_t\right]
+\int_t^T \left( b_x^* (s,t)Y_i(s)+\mathbb{E}\left[b_{x_\delta}^* (s+\delta,t+\delta)\right.
\right.\\
~~~~~~~~\left.\left.\cdot Y_i(s+\delta)1_{[0,T-\delta]^2}(s,t)\big| \mathcal{F}_s\right]+\sigma_x^* (s,t)Z_i (s,t) \right) ds
-\int_t^T Z_i(t,s) dW(s),~~t\in [0, T],\\
Y_i (t)=0,~~t\in [T, T+K],\\
Z_i (t,s)=0,~~(t,s)\in [0, T+K]^2 \backslash [0, T]^2,\\
Y_{0_i} (t) = \int_t^T \left(b_{u_i}^* (s,t)Y_i(s)+\mathbb{E}\left[b_{u_{i\delta_i}}^* (s+\delta_i,t+\delta_i)Y_i(s+\delta_i)1_{[0,T-\delta_i]^2}(s,t)\big| \mathcal{F}_s\right]\right.
\\
~~~~~~~~~\left.+\sigma_{u_i}^* (s,t)Z_i (s,t)\right) ds
-\int_t^T Z_{0_i} (t,s) dW(s),~~t\in [0, T],
\end{array}
\right.
\end{align}
such that
\begin{equation}\label{39}
\begin{aligned}
&\left(Y_{0_1} (t)+h_{1u_1}^* (t)+\mathbb{E}\left[ h_{1u_{1\delta_1}}^* (t+\delta_1)1_{[0,T-\delta_1]}(t)\big| \mathcal{F}_t\right]\right)
(u_1-u_1^* (t))\geq 0,\\
&\left(Y_{0_2} (t)+h_{2u_2}^* (t)+\mathbb{E}\left[ h_{2u_{2\delta_2}}^* (t+\delta_2)1_{[0,T-\delta_2]}(t)\big| \mathcal{F}_t\right]\right)
(u_2-u_2^* (t))\geq 0,
\end{aligned}
\end{equation}
for any $(u_1,u_2)\in \mathbf{U}_1\times\mathbf{U}_2$,~$t\in [0,T]$,~a.s.
\end{theorem}
\begin{proof}
Let $(u_1(\cdot), u_2(\cdot))\in \mathcal{A}_1\times\mathcal{A}_2$.
For all $\varepsilon\in [0,1]$, we define
\begin{equation*}
\begin{aligned}
u_1^\varepsilon (\cdot):=u_1^* (\cdot)+\varepsilon [u_1(\cdot)-u_1^* (\cdot)],~~~
u_2^\varepsilon (\cdot):=u_2^* (\cdot)+\varepsilon [u_2(\cdot)-u_2^* (\cdot)].
\end{aligned}
\end{equation*}
It should be noticed that both $\mathbf{U}_1$ and $\mathbf{U}_2$ are convex sets,
then $(u_1^\varepsilon (\cdot), u_2^\varepsilon (\cdot))$ is still in $\mathcal{A}_1\times\mathcal{A}_2$.
Suppose that $X_1^\varepsilon (\cdot)\in L_\mathbb{F}^2 (-\delta,T; \mathbb{R}) \left( \mbox{resp.}~ X_2^\varepsilon (\cdot)\in L_\mathbb{F}^2 (-\delta,T; \mathbb{R})\right)$ is the adapted
solution to SDVIE (\ref{29}) along with the controls
$(u_1^\varepsilon (\cdot), u_2^* (\cdot)) \left( \mbox{resp.}~ (u_1^*(\cdot), u_2^\varepsilon (\cdot)) \right)$.
Denote by
\begin{align}\label{40}
\xi_i^\varepsilon (t):=\frac{X_i^\varepsilon (t)-X^* (t)}{\varepsilon},~~t\in [-\delta,T],~~i=1,2,
\end{align}
and we can let $\varepsilon\mapsto 0$ in (\ref{40}) to obtain
$\xi_i^\varepsilon (\cdot)\mapsto \xi_i (\cdot)\in L_{\mathbb{F}}^{2} (-\delta,T; \mathbb{R})$ such that
\begin{align*}
\left\{\begin{array}{lcl}
\xi_i(t)=\int_{0}^{t} \left[ b_x^* (t, s)\xi_i(s)+b_{x_\delta}^* (t, s)\xi_i(s-\delta)+b_{u_i}^* (t, s)(u_i(s)-u_i^* (s))
\right.\\
~~~~~~~~~\left.+b_{u_{i\delta_i}}^* (t, s)(u_i(s-\delta_i)-u_i^*(s-\delta_i))\right] ds\\
~~~~~~~~~+\int_{0}^{t}\left[\sigma_x^* (t, s)\xi_i(s)+\sigma_{u_i}^* (t, s)(u_i(s)-u_i^* (s))
\right]   dW(s)
\\
~~~~=\psi_i^* (t)
+\int_{0}^{t} \left[ b_x^* (t, s)\xi_i(s)+b_{x_\delta}^* (t, s)\xi_i(s-\delta)\right] ds
+\int_{0}^{t} \left[\sigma_x^* (t, s)\xi_i(s) \right]  dW(s),
~~t\in [0, T],\\
\xi_i (t)=0,~~t\in [-\delta, 0],\\
u_i (t)-u_i^* (t)=0,~~t\in [-\delta_i, 0],~~i=1,2,
\end{array}
\right.
\end{align*}
where
\begin{equation*}
\begin{aligned}
\psi_i^* (t)&=\int_{0}^{t} \left[b_{u_i}^* (t, s)(u_i(s)-u_i^* (s))+b_{u_{i\delta_i}}^* (t, s)(u_i(s-\delta_i)-u_i^*(s-\delta_i)) \right] ds\\
&~~~+\int_{0}^{t}\left[\sigma_{u_i}^* (t, s)(u_i(s)-u_i^* (s))  \right]  dW(s),~~t\in [0, T].
\end{aligned}
\end{equation*}
Based on Theorem \ref{th1}, let us assume that $(Y_i(\cdot),Y_{0_i} (\cdot); Z_i(\cdot,\cdot),Z_{0_i}(\cdot,\cdot))~(i=1,2)$ is the unique adapted M-solution to ABSVIE (\ref{38}). For $i=1$, according to Theorem \ref{th4}, and by (\ref{2}) and (\ref{30}) we have that
\begin{equation*}
\begin{aligned}
&0\leq \frac{J_1(u_1^\varepsilon (\cdot),u_2^* (\cdot))-J_1(u_1^* (\cdot),u_2^* (\cdot))}{\varepsilon}\\
&~\rightarrow \mathbb{E}\left\{\int_0^T \left[h_{1x}^* (t)\xi_1(t)+h_{1x_\delta}^* (t)\xi_1(t-\delta)
+h_{1u_1}^* (t)(u_1(t)-u_1^* (t))\right.\right.\\
&~~~~\left.\left.+h_{1u_{1\delta_1}}^* (t)(u_1(t-\delta_1)-u_1^* (t-\delta_1)) \right]dt\right\}\\
&~=\mathbb{E}\left[\int_0^T h_{1x}^* (t)\xi_1 (t)dt+\int_{-\delta}^{T-\delta} h_{1x_\delta}^* (t+\delta)\xi_1(t)1_{[0,T-\delta]}(t)dt
+\int_0^T h_{1u_1}^* (t)(u_1(t)-u_1^* (t)) dt\right.\\
&~~~~\left.+\int_{-\delta_1}^{T-\delta_1} h_{1u_{1\delta_1}}^* (t+\delta_1)(u_1(t)-u_1^* (t))1_{[0,T-\delta_1]}(t) dt \right]\\
&~=\mathbb{E}\left[\int_0^T \left( h_{1x}^* (t)+\mathbb{E}\left[h_{ix_{\delta}}^* (t+\delta)1_{[0,T-\delta]}(t)\big| \mathcal{F}_t\right]  \right)\xi_1 (t)dt\right]\\
&~~~~
+\mathbb{E}\left[\int_0^T \left( h_{1u_1}^* (t)+\mathbb{E}\left[ h_{1u_{1\delta_1}}^* (t+\delta_1)1_{[0,T-\delta_1]}(t)\big| \mathcal{F}_t\right]  \right)(u_1(t)-u_1^* (t)) dt\right]\\
&~=\mathbb{E}\left[\int_0^T \psi_1^* (t)Y_1(t)dt\right]
+\mathbb{E}\left[\int_0^T \left( h_{1u_1}^* (t)+ \mathbb{E}\left[h_{1u_{1\delta_1}}^* (t+\delta_1)1_{[0,T-\delta_1]}(t)\big| \mathcal{F}_t\right] \right)(u_1(t)-u_1^* (t)) dt\right]\\
&~=\mathbb{E}\left\{\int_0^T \left[\int_{0}^{t} \left(b_{u_1}^* (t, s)(u_1(s)-u_1^* (s))
+b_{u_{1\delta_1}}^* (t, s)(u_1 (s-\delta_1)-u_1^* (s-\delta_1)) \right) ds\right] Y_1(t) dt\right\}\\
&~~~~+\mathbb{E}\left\{\int_0^T \left[\int_{0}^{t} \left(\sigma_{u_1}^* (t, s)(u_1(s)-u_1^* (s))
 \right) dW(s)\right]Y_1(t) dt\right\}\\
&~~~~ +\mathbb{E}\left[\int_0^T \left( h_{1u_1}^* (t)+\mathbb{E}\left[ h_{1u_{1\delta_1}}^* (t+\delta_1)1_{[0,T-\delta_1]}(t)\big| \mathcal{F}_t\right]  \right)(u_1(t)-u_1^* (t)) dt\right]\\
&~=\mathbb{E}\left[\int_0^T \left(\int_{t}^{T} b_{u_1}^* (s, t)Y_1(s)ds\right)(u_1(t)-u_1^* (t))dt\right]
+\mathbb{E}\left[\int_0^T \left(\int_{t}^{T}b_{u_{1\delta_1}}^* (s, t)Y_1(s)ds\right)\right.\\
&~~~~\left.\cdot(u_1 (t-\delta_1)-u_1^* (t-\delta_1))  dt\right]\\
&~~~~+\mathbb{E}\left\{\int_0^T \left[\int_{0}^{t} \left(\sigma_{u_1}^* (t, s)(u_1(s)-u_1^* (s))
 \right) dW(s)\right] \left(\mathbb{E}[Y_1(t)]+\int_0^t Z_1 (t,s) dW(s)\right)dt\right\}\\
&~~~~ +\mathbb{E}\left[\int_0^T \left( h_{1u_1}^* (t)+\mathbb{E}\left[ h_{1u_{1\delta_1}}^* (t+\delta_1)1_{[0,T-\delta_1]}(t)\big| \mathcal{F}_t\right]  \right)(u_1(t)-u_1^* (t)) dt\right]
\end{aligned}
\end{equation*}
\begin{equation*}
\begin{aligned}
&~=\mathbb{E}\left\{\int_0^T \left[\int_{t}^{T} \left(b_{u_1}^* (s, t)Y_1(s)+\sigma_{u_1}^* (s,t)Z_1 (s,t)\right) ds\right](u_1(t)-u_1^* (t))dt\right\}\\
&~~~~+\mathbb{E}\left\{\int_0^T \left[\int_{t}^{T}\left(b_{u_{1\delta_1}}^* (s, t)Y_1(s)
\right) ds\right](u_1 (t-\delta_1)-u_1^* (t-\delta_1))  dt\right\}\\
&~~~~+\mathbb{E}\left[\int_0^T \left( h_{1u_1}^* (t)+\mathbb{E}\left[ h_{1u_{1\delta_1}}^* (t+\delta_1)1_{[0,T-\delta_1]}(t)\big| \mathcal{F}_t\right] \right)(u_1(t)-u_1^* (t)) dt\right] \\
&~=\mathbb{E}\left\{\int_0^T \left[\int_{t}^{T} \left(b_{u_1}^* (s, t)Y_1(s)+\mathbb{E}\left[b_{u_{1\delta_1}}^* (s+\delta_1, t+\delta_1)Y_1(s+\delta_1)1_{[0,T-\delta_1]^2}(s,t)\big| \mathcal{F}_s\right]
\right.\right.\right.\\
&~~~\left.\left.\left. +\sigma_{u_1}^* (s,t)Z_1 (s,t)\right) ds\right](u_1(t)-u_1^* (t))dt\right\}
+\mathbb{E}\left[\int_0^T \left( h_{1u_1}^* (t)\right.\right.\\
&~~~\left.\left.+ \mathbb{E}\left[h_{1u_{1\delta_1}}^* (t+\delta_1)1_{[0,T-\delta_1]}(t) \big| \mathcal{F}_t\right] \right)(u_1(t)-u_1^* (t)) dt\right]\\
&~=\mathbb{E}\left[\int_0^T \left(Y_{0_1}(t)+\int_t^T Z_{0_1}(t,s)dW(s)+h_{1u_1}^* (t)+ \mathbb{E}\left[h_{1u_{1\delta_1}}^* (t+\delta_1)1_{[0,T-\delta_1]}(t)\big| \mathcal{F}_t \right ] \right)\right.\\
&~~~\left. \cdot(u_1(t)-u_1^* (t)) dt\right]\\
&~=\mathbb{E}\left[\int_0^T \left(Y_{0_1}(t)+h_{1u_1}^* (t)+\mathbb{E}\left [ h_{1u_{1\delta_1}}^* (t+\delta_1)1_{[0,T-\delta_1]}(t)\big| \mathcal{F}_t \right ] \right) (u_1(t)-u_1^* (t)) dt\right].
\end{aligned}
\end{equation*}
Consequently, thanks to the arbitrariness of $u_1(\cdot)\in \mathcal{A}_1$, it can be used to obtain (\ref{39}).
Similarly, we can prove the case of $i=2$ in the same way.
\end{proof}

According to Theorem \ref{th5}, the Hamiltonian function $H_i$ is defined by
\begin{align*}
&~~~~H_i(t, X^* (t), X^* (t-\delta), u_1^* (t), u_1^* (t-\delta_1), u_2^* (t), u_2^* (t-\delta_2),
Y_i (\cdot), Z_i(\cdot, t); u_i)\\
&:=-\left(Y_{0_i} (t)+h_{iu_i}^* (t)+\mathbb{E}\left[ h_{iu_{i\delta_i}}^* (t+\delta_i)1_{[0,T-\delta_i]}(t)\big| \mathcal{F}_t\right]\right) u_i\\
&=-\mathbb{E}\left[h_{iu_i}^* (t)+ h_{iu_{i\delta_i}}^* (t+\delta_i)1_{[0,T-\delta_i]}(t)
+\int_t^T \left(b_{u_i}^* (s,t)Y_i(s)
\right.\right.\\
&~~~\left.\left.+b_{u_{i\delta_i}}^* (s+\delta_i,t+\delta_i)Y_i(s+\delta_i)1_{[0,T-\delta_i]^2}(s,t)+\sigma_{u_i}^* (s,t)Z_i (s,t)  \right) ds \Big| \mathcal{F}_t\right] u_i,~~i=1,2,
\end{align*}
and then the maximum condition (\ref{39}) implies that
\begin{equation*}
\begin{aligned}
&~~~H_1(t, X^* (t), X^* (t-\delta), u_1^* (t), u_1^* (t-\delta_1), u_2^* (t), u_2^* (t-\delta_2),
Y_1 (\cdot), Z_1(\cdot, t); u_1^* (t))\\
&=\max\limits_{u_1\in\mathcal{A}_1} H_1(t, X^* (t), X^* (t-\delta), u_1^* (t), u_1^* (t-\delta_1), u_2^* (t), u_2^* (t-\delta_2),
Y_1 (\cdot), Z_1(\cdot, t); u_1),\\
&~~~H_2(t, X^* (t), X^* (t-\delta), u_1^* (t), u_1^* (t-\delta_1), u_2^* (t), u_2^* (t-\delta_2),
Y_2 (\cdot), Z_2(\cdot, t); u_2^* (t))\\
&=\max\limits_{u_2\in\mathcal{A}_2} H_2(t, X^* (t), X^* (t-\delta), u_1^* (t), u_1^* (t-\delta_1), u_2^* (t), u_2^* (t-\delta_2),
Y_2 (\cdot), Z_2(\cdot, t); u_2),
\end{aligned}
\end{equation*}
where $(Y_i (\cdot), Z_i (\cdot,\cdot))$ satisfies the following adjoint equation with respect to the optimal triple $(X^* (\cdot), u_1^* (\cdot), u_2^* (\cdot))$:
\begin{align*}
\left\{\begin{array}{lcl} Y_i (t) = h_{ix}^* (t) + \mathbb{E}\left[h_{ix_{\delta}}^* (t+\delta)1_{[0,T-\delta]}(t)\big| \mathcal{F}_t\right]
+\int_t^T \left( b_x^* (s,t)Y_i(s)+\mathbb{E}\left[b_{x_\delta}^* (s+\delta,t+\delta)
\right.\right.\\
~~~~~~~~\left. \left.\cdot Y_i(s+\delta)1_{[0,T-\delta]^2}(s,t)\big| \mathcal{F}_s\right]+\sigma_x^* (s,t)Z_i (s,t)\right) ds
-\int_t^T Z_i(t,s) dW(s),~~t\in [0, T],\\
Y_i (t)=0,~~t\in [T, T+K],\\
Z_i (t,s)=0,~~(t,s)\in [0, T+K]^2 \backslash [0, T]^2,~~i=1,2.
\end{array}
\right.
\end{align*}

\begin{remark}
Under Assumptions $(A5)$-$(A6)$, Theorem \ref{th5} establishes Pontryagin's type maximum principle for
nonzero-sum differential game systems of SDVIE (\ref{29}) for the first time.
However, it is found that if SDVIE (\ref{29}) includes the average delay, i.e., $\int_{-\delta}^{0} e^{\lambda\theta} X(s+\theta) d\theta$,
then the duality of (\ref{29}) is not a trivial extension of the classical results.
As far as we know, this problem is still open at present, and expected to be explored
in our future publications. Motivated by \cite{Meng2025,Nie2022},
we also expect to establish the maximum principle for
nonzero-sum differential game systems of SDVIEs with non-convex
control sets.
\end{remark}

In what follows, we apply the maximum principle (i.e., Theorem \ref{th5}) to study a linear-quadratic
nonzero-sum differential game problem of SDVIEs.

\begin{example}\label{ex1}
For $i=1,2$, suppose that the control domain is denoted by $\mathbf{U}_i:=\mathbb{R}$.
We consider a controlled
state process $X(\cdot)$ given by the following linear SDVIE:
\begin{align}\label{52}
\left\{\begin{array}{lcl} X(t) = \phi (t)+\int_{0}^{t} \left[a_1(t, s) X(s)+a_2(t,s) X(s-\delta)
+b_1(t,s)u_1 (s)+b_2(t,s) u_2(s)\right.\\
\ \ \ \ \ \ \ \ \ \left.+c_1(t,s) u_1(s-\delta_1)+c_2 (t,s) u_2 (s-\delta_2)\right]ds\\
\ \ \ \ \ \ \  \  \ +\int_{0}^{t}  \left[\tilde{a}_1( t, s) X(s)
+\tilde{b}_1(t,s)u_1 (s)+\tilde{b}_2(t,s)u_2(s)\right] dW(s),\ \
  t\in [0, T],\\
X(t) = \phi (t), \ \ t\in [-\delta, 0],\\
u_1(t) = \vartheta_1 (t), \ \ t\in [-\delta_1, 0],\\
u_2(t) = \vartheta_2 (t), \ \ t\in [-\delta_2, 0],
\end{array}
\right.
\end{align}
where $\phi(\cdot)\in L_{\mathbb{F}}^{2} (-\delta,T; \mathbb{R})$,
$\vartheta_i (\cdot)\in L_{\mathbb{F}}^{2} (-\delta_i,0; \mathbb{R})$,
$a_i(\cdot,\cdot),b_i(\cdot,\cdot),c_i(\cdot,\cdot),\tilde{a}_1(\cdot,\cdot),
\tilde{b}_i(\cdot,\cdot): \Delta^c \mapsto \mathbb{R}$
are given bounded measurable functions, $u_i(\cdot)$ is the control process
of Player $i$, and the admissible control set for Player $i$ is
defined by $\mathcal{A}_i:=L_{\mathbb{F}}^{2}(0,T; \mathbb{R})$, $i=1,2$.
Moreover, the constants $\delta,\delta_1,\delta_2>0$ are given finite time delays
for processes $X(\cdot),u_1(\cdot),u_2(\cdot)$, respectively,
and there exists a constant $K\geq0$ such that $0<\delta,\delta_1,\delta_2\leq K$.

Define the associated performance functional of the two players by
\begin{align}\label{53}
J_i(u_1(\cdot),u_2(\cdot))=\frac{1}{2}\mathbb{E}\left[\int_{0}^{T}
\left(q_i(t)X^2 (t)+\tilde{q}_i (t) X^2 (t-\delta)
+r_i (t)u_i^2 (t)+\tilde{r}_i (t) u_i^2 (t-\delta_i)\right)dt\right],
\end{align}
where $q_i(\cdot)$, $\tilde{q}_i (\cdot)$ are nonnegative bounded measurable functions, and
$r_i(\cdot)$, $\tilde{r}_i (\cdot)$ are positive bounded measurable functions, $i=1,2$.

It is not hard to see that for any given $\phi(\cdot)\in L_{\mathbb{F}}^{2} (-\delta,T; \mathbb{R})$,
$\vartheta_i (\cdot)\in L_{\mathbb{F}}^{2} (-\delta_i,0; \mathbb{R})~(i=1,2)$ and
$(u_1(\cdot),u_2(\cdot))\in \mathcal{A}_1 \times \mathcal{A}_2$,
there exists a unique adapted solution $X(\cdot)\in L_{\mathbb{F}}^{2}(-\delta,T; \mathbb{R})$ to
the linear SDVIE (\ref{52}). Then the linear-quadratic nonzero-sum differential game
problem for the game
system (\ref{52})-(\ref{53}) is shown as follows:
\begin{problem}\label{prob2}
Find a pair of admissible controls
$(u_1^* (\cdot), u_2^* (\cdot))\in \mathcal{A}_1\times\mathcal{A}_2$ such that
\begin{align}\label{55}
\left\{\begin{array}{lcl} J_1(u_1^* (\cdot), u_2^*(\cdot)) = \inf\limits_{u_1(\cdot)\in\mathcal{A}_1}
J_1(u_1(\cdot), u_2^* (\cdot)),\\
J_2(u_1^*(\cdot), u_2^* (\cdot)) = \inf\limits_{u_2(\cdot)\in\mathcal{A}_2} J_2(u_1^* (\cdot), u_2(\cdot)).
\end{array}
\right.
\end{align}
\end{problem}
For $i=1,2$, the Hamiltonian function $H_i$ of Problem \ref{prob2} is given by
\begin{equation*}
\begin{aligned}
&~~~~H_i(t, X^* (t), X^* (t-\delta), u_1^* (t), u_1^* (t-\delta_1), u_2^* (t), u_2^* (t-\delta_2),
Y_i (\cdot), Z_i(\cdot, t); u_i)\\
&:=-\left(Y_{0_i} (t)+r_i(t) u_i^* (t)+\tilde{r}_i(t+\delta_i)  u_i^* (t)1_{[0,T-\delta_i]}(t)\right) u_i,
\end{aligned}
\end{equation*}
and the corresponding adjoint equation has the following form:
\begin{align*}
\left\{\begin{array}{lcl} Y_i (t) = \left(q_{i} (t) +\tilde{q}_{i} (t+\delta)1_{[0,T-\delta]}(t)\right) X^*(t)
+\int_t^T \left( a_1 (s,t)Y_i(s)+a_2 (s+\delta,t+\delta)
\right.\\
~~~~~~~~\left.\cdot\mathbb{E}\left[ Y_i(s+\delta)\big| \mathcal{F}_s\right] 1_{[0,T-\delta]^2}(s,t)+\tilde{a}_1 (s,t) Z_i (s,t) \right) ds
-\int_t^T Z_i(t,s) dW(s),~~t\in [0, T],\\
Y_i (t)=0,~~t\in [T, T+K],\\
Z_i (t,s)=0,~~(t,s)\in [0, T+K]^2 \backslash [0, T]^2,
\end{array}
\right.
\end{align*}
which admits a unique adapted M-solution $(Y_i (\cdot), Z_i (\cdot,\cdot))\in\mathcal{M}^2 [0,T+K]$
according to Theorem \ref{th1}.

By Theorem \ref{th5}, the Nash equilibrium point
$(u_1^* (\cdot), u_2^* (\cdot))\in \mathcal{A}_1\times\mathcal{A}_2$ of Problem \ref{prob2} is given by
\begin{align}\label{54}
\left\{\begin{array}{lcl}
u_1^*(t)= -\left[r_1(t)+\tilde{r}_1 (t+\delta_1)1_{[0,T-\delta_1]}(t)\right]^{-1} Y_{0_1}(t),\\
u_2^*(t)= -\left[r_2(t)+\tilde{r}_2 (t+\delta_2)1_{[0,T-\delta_2]}(t)\right]^{-1} Y_{0_2}(t),\ \ t\in [0, T],
\end{array}
\right.
\end{align}
where $(Y_{0_i} (\cdot),Z_{0_i}(\cdot,\cdot))$ solves the following equation:
\begin{align*}
Y_{0_i} (t) &= \int_t^T \left(b_{i} (s,t)Y_i(s)+c_i (s+\delta_i,t+\delta_i)\mathbb{E}\left[Y_i(s+\delta_i)\big| \mathcal{F}_s\right] 1_{[0,T-\delta_i]^2}(s,t)
+\tilde{b}_i (s,t)Z_i (s,t)\right) ds\\
&~~~-\int_t^T Z_{0_i} (t,s) dW(s),~~t\in [0, T],~~i=1,2.
\end{align*}

\begin{proposition}\label{pro1}
For the linear-quadratic
nonzero-sum differential game problem of SDVIE
(\ref{52})-(\ref{55}), a Nash equilibrium point
$(u_1^*(\cdot), u_2^*(\cdot))$ is shown by (\ref{54}).
\end{proposition}
\end{example}

\section*{Acknowledgements}
The authors would like to thank Professor Fuke Wu for his constructive suggestions
which dramatically improved the presentation of this manuscript.
This work is supported by the Provincial Natural Science Foundation of Hunan
(Grant Nos. 2023JJ40203, 2023JJ30642) and the National Natural Science Foundation of China
(Grant Nos. 12101218, 12371141).








\end{document}